\documentclass{aptpub}

\usepackage[T1]{fontenc}
\usepackage{lmodern}
\usepackage[english]{babel}
\usepackage{amssymb}
\usepackage{amsmath}
\usepackage{theorem}
\usepackage{graphicx}
\usepackage{bbm}
\usepackage{amsmath,amstext}  
\usepackage{amsfonts,amssymb} 
\usepackage[english]{babel}
\usepackage{latexsym,amsfonts,amsmath,amssymb,amsbsy}
\usepackage{graphics, graphicx}
\usepackage[mathscr]{eucal}
\usepackage{epsfig}
\usepackage{epsf}
\usepackage{color}
\usepackage{pifont}

\newcommand{\R}{\mathbb{R}}
\renewcommand{\N}{\mathbb{N}}
\newcommand{\PP}{\mathbf{P}}
\renewcommand{\E}{\mathbf{E}}
\newcommand{\1}{\mathbbm{1}}

\newtheorem{Theorem}{Theorem}[section]
\newtheorem{Proposition}[Theorem]{Proposition}
\newtheorem{Lemma}[Theorem]{Lemma}

\newtheorem{Definition}[Theorem]{Definition}
\newtheorem{Remark}[Theorem]{Remark}

\newtheorem{Assumption}[Theorem]{Assumption}

\authornames{A. Brandejsky, B. de~Saporta, F. Dufour} 
\shorttitle{Exit time for PDMP} 

\begin{document}

\title{Numerical methods for the exit time of a piecewise-deterministic Markov process \footnote{This work was supported by ARPEGE program of the French National Agency of Research (ANR), project ''FAUTOCOES'', number ANR-09-SEGI-004.}}


\authorone[INRIA Bordeaux Sud-ouest, team CQFD, France.]{Adrien Brandejsky\footnote{The authors gratefully acknowledge EADS Astrium for its financial support.}}

\authorone[Universit\'e de Bordeaux, GREThA, CNRS, UMR 5113,\\ IMB, CNRS, UMR 5251, and INRIA Bordeaux Sud-ouest, team CQFD, France.]{Beno\^\i te de~Saporta}

\authorone[Universit\'e de Bordeaux, IMB, CNRS, UMR 5251,\\ and INRIA Bordeaux Sud-ouest, team CQFD, France.]{Fran\c cois Dufour}
\addressone{INRIA Bordeaux Sud-ouest, 351 cours de la Lib\'eration, F33405 Talence, France}

\begin{abstract}\quad
We present a numerical method to compute the survival function and the moments of the exit time for a piecewise-deterministic Markov process (PDMP). Our approach is based on the quantization of an underlying discrete-time Markov chain related to the PDMP. The approximation we propose is easily computable and is even flexible with respect to the exit time we consider. We prove the convergence of the algorithm and obtain bounds for the rate of convergence in the case of the moments. An academic example and a model from the reliability field illustrate the paper.
\end{abstract}



\newpage
\tableofcontents

\newpage

\section*{Introduction}

The aim of this paper is to propose a practical numerical method to approximate the survival function and the moments of the exit time for a piecewise-deterministic Markov process thanks to the quantization of a discrete-time Markov chain naturally embedded within the continuous-time process.\\

Piecewise-deterministic Markov processes (PDMP's) have been introduced by M.H.A. Davis in \cite{davis} as a general class of stochastic models. PDMP's are a family of Markov processes involving deterministic motion punctuated by random jumps. The motion depends on three local characteristics namely the flow $\Phi$,  the jump rate $\lambda$ and the transition measure $Q$,which specifies the post-jump location. Starting from the point $x$, the motion of the process follows the flow $\Phi(x,t)$ until the first jump time $T_{1}$, which occurs either spontaneously in a Poisson-like fashion with rate $\lambda(\Phi(x,t))$ or when the flow $\Phi(x,t)$ hits the boundary of the state space. In either case, the location of the process at the jump time $T_{1}$, is selected by the transition measure $Q(\cdot,\Phi(x,T_{1}))$ and the motion restarts from this new point $X(T_{1})$ denoted by $Z_{1}$. We define similarly the time $S_{2}$ until the next jump, the next jump time is $T_{2}=T_{1}+S_{2}$, the next post-jump location $Z_{2}=X(T_{2})$ and so on. Thus, associated to the PDMP we have discrete-time Markov chains $(Z_{n},T_{n})_{n\in \N}$, given by the post-jump locations and the jump times, and $(Z_{n},S_{n})_{n\in \N}$, given by the post-jump locations and the inter-jump times. A suitable choice of the state space and the local characteristics $\Phi$, $\lambda$ and $Q$ provides stochastic models covering a great number of problems of operations research for example see \cite{chiquet}, \cite{davis} and the corrosion model presented in this paper.\\

Numerical computation of the moments of the exit time for a Markov process has been studied by K. Helmes, S. R\"ohl and R.H. Stockbridge in \cite{LP}. Starting from an assumption related to the generator of the process, they derive a system of linear equations satisfied by the moments. In addition to these equations, they include finitely many Hausdorff moment conditions that are also linear constraints. This optimization problem is a standard linear programming problem for which many efficient softwares are available. J.-B. Lasserre and T. Prieto-Rumeau introduced in \cite{SDP} a similar method but they improved the efficiency of the algorithm by replacing the Hausdorff moment conditions with semidefinite positivity constraints of some moment matrices. Nevertheless, their approach cannot be applied to PDMP's because the assumption related to the generator of the process is generally not satisfied. In \cite{davis} section 33, M.H.A. Davis gives an iterative method to compute the mean exit time for a PDMP but his approach involves solving a large set of ODEs whose forms are very problem specific, depending on the behaviour of the process at the boundary of the state space. Besides, and in the context of applications to reliability, it seems important to study also the distribution of the exit time.\\

There exists an extensive literature on quantization methods for random variables and processes. The interested reader may for instance consult \cite{gray}, \cite{quantif} and the references within. Quantization methods have been developed recently in numerical probability or optimal stochastic control with applications in finance (see e.g. \cite{quantif1}, \cite{quantif2}, \cite{BoutonPags} and \cite{quantif}). The quantization of a Markov chain $(\Theta_n)_{n\in \N}$ consists in finding, for each $n$, an optimally designed discretization of the state space of $\Theta_n$ providing the best possible $L^{p}$ approximation by a random variable $\widehat{\Theta}_n$ taking its values in a grid $\Gamma_{n}$ of finite and fixed size as well as transition measure of the quantized chain $(\widehat{\Theta}_n)_{n\in\N}$. As explained for instance in \cite{quantif}, section 3, provided that the Markov kernel is Lipschitz, bounds for the rate of $L^{p}$ convergence of the quantized process towards the original process are obtained.\\

In the present work, we consider a PDMP $\left(X_{t}\right)_{t\geq 0}$ with state space $E$ and we present approximation methods to compute the moments and the survival function of the exit time from a set denoted $U\subset E$ given the fact that it happens before the~$N^{\text{th}}$ jump of the PDMP denoted by $T_{N}$.
Roughly speaking, we estimate the moments and the survival function for $\tau\wedge T_{N}$.
In our approach, the first step consists in expressing the $j$-th moment (respectively the survival function) as the last term of some sequence $(p_{k,j})_{k\leq N}$ (respectively $(p_{k})_{k\leq N}$) satisfying a recursion $p_{k+1,j}=\psi(p_{k,j})$ (respectively $p_{k+1}=\psi(p_{k})$) specifically built within our paper. \\

In this context, a natural way to deal with these problems is to follow the idea developed in \cite{arret optimal francois benoite} namely to write the recursions in terms of an underlying discrete-time Markov chain and to replace it by its quantized approximation. The definitions of $(p_{k,j})_{k}$ and $(p_{k})_{k}$ involve some discontinuities related to indicator functions but as in \cite{arret optimal francois benoite}, we show that they happen with small enough probability. However, an important feature that distinguishes the present work from \cite{arret optimal francois benoite} and which prevents a straightforward application of the ideas developed within, is that an additional important difficulty appears in the definition of the sequences $(p_{k,j})_{k}$ and $(p_{k})_{k}$. Indeed, the mapping $\psi$ such that $p_{k+1,j}=\psi(p_{k,j})$ and $p_{k+1}=\psi(p_{k})$ is not Lipschitz continuous. One of the main results of this paper is to overcome this difficulty by deriving new and important properties of the Markov chain $(Z_{n},T_{n})_{n\in \N}$, combined to a sharp feature of the quantization algorithm. We are able to prove the convergence of the approximation scheme. Moreover, in the case of the moments, we even obtain bounds for the rate of convergence. It is important to stress that these assumptions are quite reasonable with regards to the applications.\\

An important advantage of our method is that it is flexible. Indeed, as pointed out in \cite{quantif1}, a quantization based method is ``obstacle free'' which means, in our case, that it produces, once and for all, a discretization of the process independently of the set $U$. Consequently, the approximation schemes for both the moments and the distribution of the exit time are flexible w.r.t. to $U$. Indeed, if we are interested in the exit time from a new set $U'$, it will be possible, provided $U'$ satisfies the same assumptions as $U$, to obtain in a very simple way the moments and the distribution of this new exit time. Indeed, the quantization grids are only computed once, stored off-line and may therefore serve many purposes. \\

The paper is organized as follows. We first recall the definition of a PDMP and state our assumptions. In Section \ref{section-exit-time}, we introduce the moments and the distribution problems and present recursive methods to solve them. Section \ref{section-approximation} contains the main contribution of this paper namely the approximation schemes, the proofs of convergence and bounds for the rates of convergence. Eventually, two numerical examples are developed in Section \ref{section-examples} and the advantages of our approach are discussed in Section \ref{conclusion}.

\section{Definitions and assumptions}

For any metric space $X$, we denote $\mathcal{B}(X)$ its Borel $\sigma$-field and $B(X)$ the set of real-valued, bounded and measurable functions defined on~$X$. For $a,b \in \R$, denote $a\wedge b=\min(a,b)$ and $a\vee b=\max(a,b)$.

\subsection*{Definition of a PDMP}
In this first section, let us define a piecewise-deterministic Markov process and introduce some general assumptions. Let $M$ be a finite set called the set of the modes that will represent the different regimes of evolution of the PDMP ($M$ is supposed to be a finite space although it could be countable), for each $m\in M$, the process evolves in $E_m$, an open subset of $\R^{d(m)}$ (where $d:M\rightarrow \mathbb{N}^*$). Let $$E=\left\{(m,\xi),m\in M,\xi\in E_m \right\}.$$
This is the state space of the process $(X_t)_{t\in \mathbb{R}^+}=(m_t,\xi_t)_{t\in \mathbb{R}^+}$. Let $\partial E$ be its boundary and $\overline{E}$ its closure and for any subset $Y$ of $E$, $Y^{c}$ denotes its complement.\\

Define on $E$ the following distance, for $x=(m,\xi)$ and $x'=(m',\xi') \in E$,
\begin{equation}\label{def-dist}
|x-x'|=\left\{
\begin{array}{ll}
+\infty & \text{ if } m\neq m',\\
|\xi-\xi'|& \text{ otherwise. }
\end{array}
\right.
\end{equation}
Moreover, for any $x\in E$ and $Y\subset E$, denote $d(x,Y)$ the distance between the point $x$ and the set $Y$ i.e. $d(x,Y)=\inf_{y\in Y}|x-y|$. \\

A PDMP is defined by its local characteristics $(\Phi_{m},\lambda_{m},Q_{m})_{m\in M}$ :
\begin{itemize}
\item{For each $m\in M$, $\Phi_m : \mathbb{R}^{d(m)}\times \mathbb{R}\rightarrow \mathbb{R}^{d(m)}$ is a continuous function called the flow in mode $m$. For all $t\in \mathbb{R}$, $\Phi_m(\cdot,t)$ is an homeomorphism and $t\rightarrow \Phi_m(\cdot,t)$ is a group i.e. for all $\xi\in \mathbb{R}^{d(m)}$, $\Phi_m(\xi,t+s)=\Phi_m(\Phi_m(\xi,s),t)$.
    For all $x=(m,\xi)\in E$, define now the deterministic exit time from $E$ :
    $$t^*(x)=
    \inf\{t>0 \text{ such that } \Phi_m(\xi,t)\in\partial E_m\}$$
    We use here and throughout the whole paper the convention $\inf \emptyset = + \infty$.
    }
\item{For all $m\in M$, the jump rate $\lambda_{m} : \overline{E}_{m}\rightarrow \mathbb{R}^+$ is measurable and satisfies :
$$\forall (m,\xi)\in E\text{,  }\exists \epsilon >0\text{ such that }\int_0^\epsilon \lambda_{m}(\Phi_m(\xi,t))dt< +\infty.$$
}
\item{For all $m\in M$, $Q_{m}$ is a Markov kernel on $(\mathcal{B}(\overline{E}),\overline{E}_{m})$ which satisfies :
$$\forall \xi \in \overline{E}_{m}\text{,  }Q_{m}(E\backslash\{(m,\xi)\},\xi)=1.$$
}
\end{itemize}
From these characteristics, it can be shown (see \cite{davis}) that there exists a filtered probability space $(\Omega,\mathcal{F},{\mathcal{F}_t},(\mathbf{P}_x)_{x\in E})$ on which a process $(X_t)_{t\in \mathbb{R}^+}$ is defined. Its motion, starting from a point $x\in E$, may be constructed as follows.
Let $T_1$ be a nonnegative random variable with survival function :
$$\PP_x(T_1>t)=
\left\{\begin{array}{ll}
e^{-\Lambda(x,t)} & \text{if } 0\leq t<t^*(x), \\
0 & \text{if }t\geq t^*(x),
\end{array}\right.$$
where for $x=(m,\xi)\in E$ and $t\in [0,t^*(x)]$,
$$\Lambda(x,t)=\int_0^t\lambda_{m}(\Phi_m(\xi,s))ds.$$

One then chooses an $E$-valued random variable $Z_1$ with distribution $Q_{m}(\cdot,\Phi_m(\xi,T_1))$. The trajectory of $X_t$ for $t\leq T_1$ is :
$$X_t=\left\{
\begin{array}{ll}
(m,\Phi_m(\xi,t))&\text{ if }t<T_1,\\
Z_1&\text{ if }t=T_1.
\end{array}\right.
$$
Starting from the point $X_{T_1}=Z_1$, one then selects in a similar way $S_{2}=T_2-T_1$ the time between $T_{1}$ and the next jump, $Z_2$ the next post-jump location and so on.\\
M.H.A. Davis shows (see \cite{davis}) that the process so defined is a strong Markov process $(X_t)_{t\geq 0}$ with jump times $(T_n)_{n\in {\mathbb{N}}}$ (with $T_0=0$). The process $(\Theta_n)_{n\in\N}=(Z_n,T_n)_{n\in\N}$ where $Z_n=X_{T_n}$ is the post-jump location and $T_{n}$ is the $n$-th jump time is clearly a discrete-time Markov chain. Besides, we denote  $S_n=T_n-T_{n-1}$ and $S_0=0$ the inter-jump times.\\
The following assumption about the jump-times is standard (see for example \cite{davis}, section 24) :
\begin{Assumption}\label{hyp-Tk_goes_to_infty}
For all $(x,t)\in E\times \mathbb{R}^+$, $\E_x\left[\sum_k \mathbbm{1}_{\{T_k<t\}}\right]<+\infty$.
\end{Assumption}
It implies that $T_k\rightarrow +\infty$ a.s. when $k\rightarrow +\infty$.\\

\paragraph{}
For notational convenience, any function $h$ defined on $E$ will be identified with its component functions $h_m$ defined on $E_m$. Thus, one may write $$h(x)=h_m(\xi) \text{ when } x=(m,\xi)\in E.$$
We also define a generalized flow $\Phi : E\times\mathbb{R}^+\rightarrow E$ such that $$\Phi(x,t)=(m,\Phi_m(\xi,t)) \text{ when } x=(m,\xi)\in E.$$

\subsection*{Notation}
\paragraph{}For any function $w$ in $B(\overline{E})$, introduce the following notation 
\begin{equation*}
Qw(x)=\int_E w(y)Q(dy,x) \text{, } C_w=\sup_{x\in \overline{E}}|w(x)|
\end{equation*}
and for any Lipschitz continuous function $w$ in $B(\overline{E})$, denote $[w]$ its Lipschitz constant:
$$[w]=\sup_{x\neq y \in \overline{E}}\frac{|w(x)-w(y)|}{|x-y|}$$
with the convention $\frac{1}{\infty}=0$.

\begin{Remark}\label{RqLip_f_fm}
For $w\in B(\overline{E})$ and from the definition of the distance on $E$, one has $[w]=\text{sup}_{m\in M}[w_m]$.
\end{Remark}

\section{Exit time}\label{section-exit-time}

For all $m\in M$, let $U_m$ be a Borel subset of $E_m$, let $U=\left\{(m,\xi),m\in M,\xi\in U_m \right\}$. We are interested in the exit time from $U$ denoted by~$\tau$ :
\begin{equation}\label{DefTau}
\tau=\text{inf} \left\{ s \geq 0 \text{ such that } X_s \not\in U\right\}.
\end{equation}
Denote $\mu$ the distribution of the initial state of the process $Z_{0}$. Since the present paper concerns numerical computations, the following assumption appears natural.
\begin{Assumption}\label{hyp-tau-fini}
The process starts in $U$ and eventually leaves it almost surely i.e. the support of $\mu$ is included in $U$ and $\PP_{\mu}\left(\tau<+\infty\right)=1$.
\end{Assumption}
The aim of this paper is to provide approximation schemes for its survival function and its moments. Our method has a high practical interest because it will provide numerical approximations as soon as the process can be simulated. Our approach is based on a recursive computation using the underlying discrete-time Markov chain $(Z_n,T_n)_{n\in\N}$. Therefore, we will study $\tau\wedge T_N$ rather than $\tau$ for some $N\in \N$ called the computation horizon. Indeed, thanks to Assumption \ref{hyp-Tk_goes_to_infty}, when $N$ goes to infinity, one has
$$\tau\wedge T_N\rightarrow \tau \qquad \text{ $\PP_\mu$ a.s.}$$
One may approximate $\tau$ by $\tau\wedge T_{N}$ if $N$ is chosen such that $\PP_{\mu}(\tau > T_N)$ be small enough (the choice of $N$ will be discussed in section \ref{ChoixHorizonCalcul}) because the evolution of the process beyond $T_{N}$ will have little impact on the law or the moments of the exit time.
The rest of this section presents the two problems we are interested in and describes recursive methods to solve them.\\

\begin{Definition}\label{Defu*}
Let us introduce $u^*(x)$ for all $x\in U$ as the time for the flow starting from the point $x$ to exit from $U$ :
\begin{equation*}
u^*(x)=\inf\left\{s\geq 0 \text{ such that } \Phi(x,s)\not\in U\right\}
\end{equation*}
\end{Definition}
We now introduce some technical assumptions that will be in force throughout the whole paper. The first three ones will be crucial while the two last ones can be made without loss of generality.

\begin{Assumption}\label{hyp_u*_lip}The function $u^{*}$ is
\begin{description}
\item[a.]{Lipschitz continuous,}
\item[b.]{bounded by $C_{u^{*}}$.}
\end{description}
\end{Assumption}

\begin{Assumption}\label{hyp-U-convex} For all $m\in M$, the set $U_{m}$ is convex.
\end{Assumption}

\begin{Assumption}\label{hyp-proba-U-apha}
For $\alpha>0$, let $U^{\alpha}=\left\{x\in E \text{ such that } d(x,\partial U)\leq \alpha\right\}$. There exists $C>0$ and $\beta>0$ such that for all $k\in \{0,...,N\}$, $\PP_{\mu}(Z_{k}\in U^{\alpha})\leq C \alpha^{\beta}$.
\end{Assumption}
\begin{Remark} This technical condition can be checked in most of the applications. We will see, in the examples developed in Section \ref{section-examples}, how it can be derived quite generally when $Z_{k}$ has a bounded density. Moreover, it could be replaced by the following one, similar to an hypothesis introduced by M.H.A. Davis in \cite{davis} (Section 24) and presented as quite general in applications : there exists $\epsilon>0$ such that for all $x\in U$, $Q(U^\epsilon,x)=0$ where $U^\epsilon=\left\{x\in E \text{ such that } d(x,\partial U)\leq\epsilon\right\}$ i.e. for all $k\in \{0,...,N\}$, $\PP_{\mu}(Z_{k}\in U^{\epsilon})=0$.
\end{Remark}

\begin{Assumption}\label{hypNonRetourU}
The process cannot go back to $U$ once it has left it i.e. $\forall z \in U^{c}, \PP_z(\forall t \geq 0, X_t\in U)=0.$
\end{Assumption}
\begin{Assumption}\label{hyp_t*_bnd}
The function $t^{*}$ is bounded by $C_{t^{*}}$.
\end{Assumption}

In our discussion, Assumption \ref{hypNonRetourU} does not imply any loss of generality and Assumption \ref{hyp_t*_bnd} stems from Assumption \ref{hyp_u*_lip}.b. Indeed, if any of the two previous assumptions is not satisfied by the process $(X_{t})_{t\in\R^{+}}$, we introduce the process \textit{killed at time $\tau$} denoted $(\widetilde{X}_{t})_{t\in\R^{+}}$ and defined by:
\begin{equation*}
\widetilde{X}_{t=}\left\{\begin{array}{ll}
X_{t}& \text{ for $t<\tau$,}\\
\Delta &\text{ for $t\geq\tau$.}
\end{array}\right.\end{equation*}
where $\Delta$ denotes a cemetery state. The state space of the killed process is $\widetilde{E}=U\cup \{\Delta\}$ and Assumption \ref{hypNonRetourU} is fulfilled since the killed process remains in $\Delta$ after leaving $U$. In addition, $\tilde{t}^{*}$, the deterministic exit time from $\widetilde{E}$ for the killed process equals $u^{*}$ that is bounded and Lipschitz continuous according to Assumption \ref{hyp_u*_lip}.

\subsection{Distribution}

The first goal of this paper is to compute an approximation of the law of the exit time $\tau$. More precisely, we intend to approximate $\mathbf{P}_{\mu}(\tau>s\big|\tau\leq T_N)$ for $s>0$.

\bigskip

Our approach has a huge practical interest because we will see that, after some initial computations, any value of the survival function of $\tau$ may be quickly obtained. More importantly, our approach is even flexible with respect to $U$ in the sense that the survival function of the exit time $\tau'$ from a new set $U'\subset U$ will also be directly available (provided that Assumptions \ref{hyp_u*_lip} to \ref{hypNonRetourU} are still fulfilled by $U'$).\\

\begin{Definition}For all $s>0$, define as follows the sequences $(p_k(s))_{k\geq 0}$, $(q_k)_{k\geq 0}$ and $(r_k(s))_{k\geq 0}$ 
$$\left\{\begin{array}{ll}
p_k(s)&=\mathbf{P}_{\mu}(\tau>s\big|\tau\leq T_k),\\
q_k&=\mathbf{P}_{\mu}(\tau\leq T_k),\\
r_k(s)&=\mathbf{P}_{\mu}(\{\tau>s\}\cap\{T_{k}<\tau\leq T_{k+1}\}).
\end{array}\right.$$
\end{Definition}
\begin{Remark}The conditional probability $p_k(s)$ does not exist when $q_k=0$. We then choose to extend the sequence by setting $p_k(s)=0$.
\end{Remark}
\noindent Our objective is to approximate $p_N(s)$ where $N$ represents the computation horizon. The following proposition provides a recursion for the sequence $(p_k)_{k\leq N}$, pointing out that $p_N$ may be computed as soon as the sequences $(q_k)_{k\leq N}$ and $(r_k)_{k\leq N-1}$ are known.

\begin{Proposition}\label{rec_pk}Under Assumption \ref{hyp-tau-fini}, for all $k\in\N$, $s>0$, $p_0(s) =0$ and
\begin{equation*}
p_{k+1} (s)=
\left\{\begin{array}{ll}
\frac{p_k(s)q_k+r_k(s)}{q_{k+1}}, &\qquad\text{if $q_{k+1}\neq 0$}\\
0 & \qquad \text{otherwise.}
\end{array}\right.
\end{equation*}
\end{Proposition}

\begin{proof} First, recall that $T_{0}=0$ so that one has $p_0=0$ since the process starts in $U$ according to Assumption \ref{hyp-tau-fini}. Then, let $k\in\N$ such that $q_{k+1}\neq 0$ and notice that $\{\tau\leq T_{k+1}\}=\{\tau\leq T_k\}\cup\{T_k<\tau\leq T_{k+1}\}$, one has
\begin{align*}
p_{k+1}(s)&=\frac{\mathbf{P}_{\mu}(\{\tau>s\}\cap\{\tau\leq T_{k+1}\})}{\mathbf{P}_\mu(\tau\leq T_{k+1})}\\
&=\frac{\mathbf{P}_{\mu}(\{\tau>s\}\cap\{\tau\leq T_{k}\})+\mathbf{P}_{\mu}(\{\tau>s\}\cap\{T_k<\tau\leq T_{k+1}\})}{q_{k+1}}\\
&=\frac{p_k(s)q_{k}+r_k(s)}{q_{k+1}}
\end{align*}
showing the results.\hfill$\Box$
\end{proof}

Now, before turning to computations, let us present the second problem we are interested in.

\subsection{Moments}\label{DefRec}

Our second goal is to approximate the moments of the exit time from $U$ i.e. for all $j\in\N$, we are interested in $\E_{\mu}[\tau^{j}\big|\tau\leq T_{N}]$. This is a very classical problem and some results are already available. First, it is possible to use a Monte Carlo method and we will point out why the method we propose is more efficient and flexible. Furthermore, K. Helmes, S. R\"ohl and R.H. Stockbridge introduced in \cite{LP} a numerical method for computing the moments of the exit time based on linear programming. J.-B. Lasserre and T. Prieto-Rumeau improved this method in \cite{SDP} by using semidefinite positivity moment conditions. These methods are quite efficient but they require an assumption related to the generator of the process which is generally not fulfilled by the PDMP. The method we are introducing now is based on the use of the Markov chain $(\Theta_{n})_{n\in\N}=(Z_{n},T_{n})_{n\in\N}$ associated to the continuous-time process $(X_{t})_{t\in\R^{+}}$. 

\begin{Definition}For all $j\in \N$, introduce the sequences $(p_{k,j})_{k\geq 0}$ and $(r_{k,j})_{k\geq 0}$ defined as follows
$$\left\{\begin{array}{ll}
p_{k,j}&=\mathbf{E}_{\mu}\left[\tau^{j}\big|\tau\leq T_k\right],\\
r_{k,j}&=\mathbf{E}_{\mu}\left[\tau^{j}\1_{\{T_{k}<\tau\leq T_{k+1}\}}\right].
\end{array}\right.$$
\end{Definition}

\noindent Our objective is to approximate $p_{N,j}$ where $N$ still represents the computation horizon. Similarly to the previous section, the sequence $(p_{k,j})_{k\leq N}$ satisfies a recursion which parameters are the sequences $(q_k)_{k\leq N}$, previously introduced, and $(r_{k,j})_{k\leq N-1}$.

\begin{Proposition}\label{rec_pkj}Under Assumption \ref{hyp-tau-fini}, one has for all $k,j\in\N$, $p_{0,j} =0$ and
\begin{equation*}
p_{k+1,j} =
\left\{\begin{array}{ll}
\frac{p_{k,j}q_k+r_{k,j}}{q_{k+1}}, &\qquad\text{if $q_{k+1}\neq 0$}\\
0 & \qquad \text{otherwise.}
\end{array}\right.
\end{equation*}
\end{Proposition}

\begin{proof} The proof is similar to the previous one (see proposition \ref{rec_pk}).\hfill
$\Box$
\end{proof}

Before turning to the approximation method itself, let us discuss the crucial question of the computation horizon.

\subsection{The computation horizon}\label{ChoixHorizonCalcul}

In this paragraph, let us study more precisely the construction of the process $(X_t)$ in order to obtain some results concerning the jump times $(T_k)_{k\in \mathbb{N}}$. For this purpose, we introduce, in this section only two additional hypothesis. 
\begin{Assumption}\label{hyp-lambda-bornŽ} The jump rate $\lambda$ is bounded by $C_{\lambda}$.
\end{Assumption}
\begin{Assumption}\label{Hyp-Davis-Aepsilon}
There exists $\epsilon>0$ such that for all $x\in E$, $Q(x,A_\epsilon)=1$ where $A_\epsilon=\left\{x\in E \text{ such that } t^*(x)\geq\epsilon\right\}$. Roughly speaking, the jumps cannot send the process too close to the boundary of $E$.
\end{Assumption}
Assumption \ref{hyp-lambda-bornŽ} is satisfied in a large majority of applications ; Assumption \ref{Hyp-Davis-Aepsilon} is quite general too and was introduced by Davis in \cite{davis}, section 24.\\

 Let $(\Omega,\mathcal{A},\mathbf{P})$ be a probability space on which is defined a sequence $(\Pi_k)_{k\in \mathbb{N}}$ of independent random variables with uniform distribution on $[0;1]$. Let $x=(m,\xi)\in E$ and $\omega\in \Omega$ and let us focus on the construction of the trajectory $\{X_t(\omega)\text{, } t>0\}$ of the process starting from point $x$. Let
$$F(t,x)=
\left\{\begin{array}{ll}
1 & \text{if } t\leq0,\\
\text{exp}\left(-\int_0^t\lambda(m,\Phi_m(\xi,s))ds\right) & \text{if } 0\leq t<t^*(x), \\
0 & \text{if }t\geq t^*(x).
\end{array}\right.$$
It is the survival function of the first jump time $T_1$. Define then its generalized inverse:
$$\Psi(u,x)=
\left\{\begin{array}{ll}
\text{inf}\{t\geq 0 : F(t,x)\leq u\},&\\
+\infty & \text{if the above set is empty.}
\end{array}\right.
$$
Let then $S_1(\omega)=T_1(\omega)=\Psi(\Pi_1(\omega),x)$ and for all $t<T_1(\omega)$,
$$X_t(\omega)=(m,\Phi_m(\xi,t)).$$
If $T_1(\omega)<+\infty$, choose $X_{T_1}$ with distribution $Q(.,\Phi_m(\xi,T_1))$. Assume the trajectory is constructed until time $T_k$. If $T_k(\omega)<+\infty$, let
\begin{align*}
&S_{k+1}(\omega)=\Psi(\Pi_k(\omega),X_{T_k}),\\
&T_{k+1}(\omega)=T_{k}(\omega)+S_{k+1}(\omega).
\end{align*}
If $T_{k+1}(\omega)<+\infty$, choose $X_{T_{k+1}}$ with distribution $Q(.,\Phi_{m_{T_k}}(\xi_{T_k},S_{k+1}))$. The trajectory is finally constructed by induction.\\
With the same notations as above, we may then state the following lemma :
\begin{Lemma}\label{LemmeSkTilde}
Let $H$ be a survival function such that for all $t\in \mathbb{R}$ and for all $x\in E$, $H(t)\leq F(t,x)$. There exists a sequence of independent random variables $(\widetilde{S}_k)_{k\in \mathbb{N}}$ with distribution $H$ and such that
$$\forall K\in \mathbb{R} \text{, }\forall N\in \mathbb{N}\text{, }\mathbf{P}_{\mu}(T_N<K)\leq \mathbf{P}_{\mu}(\widetilde{T}_N<K).$$
where $\widetilde{T}_N=\sum_{k=0}^N\widetilde{S}_k$.
\end{Lemma}

\begin{proof}
Let $H$ be such a survival function and let $\widetilde{\Psi}$ be its generalized inverse i.e.
$$\widetilde{\Psi}(u)=
\left\{\begin{array}{ll}
\text{inf}\{t\geq 0 : H(t)\leq u\},&\\
+\infty & \text{if the above set is empty.}
\end{array}\right.
$$
The assumption made on $H$ yields for all $x\in E$, $\widetilde{\Psi}(u)\leq \Psi(u,x)$. Let for all $k\in \mathbb{N}$ and for all $\omega\in\Omega$,
$$\widetilde{S}_k(\omega)=\widetilde{\Psi}(\Pi_k(\omega)).$$
Notice that we are using the same $\Pi_k$ as in the definition of $S_k$, allowing us to write that $\widetilde{S}_k\leq S_k$ a.s. and therefore $\widetilde{T}_k\leq T_k$ a.s. The result follows.
\hfill$\Box$
\end{proof}

Similarly to M.H.A. Davis (Section 33 in \cite{davis}), we approximated $\tau$ by $\tau\wedge T_N$ since $\tau\wedge T_N\rightarrow \tau$ as $N\rightarrow +\infty$ thanks to Assumption \ref{hyp-Tk_goes_to_infty}. It is therefore necessary to choose $N$ large enough such that $\mathbf{P}_{\mu}(T_N<\tau)$ be small. It is tough to estimate this probability for a general process because the links between $\tau$ and the jump times are largely problem-dependant. For instance, the geometry of $U$ can be very complex. Therefore, $N$ will generally be estimated through simulations. Indeed, one may compute $\mathbf{P}_{\mu}(T_N<\tau)$ for some fixed $N$ thanks to a Monte-Carlo method and increase the value of $N$ until this probability becomes small enough.\\
However, we introduce an other method to bound this probability that may prove useful in applications. First, notice that, for any $K>0$,
$$\{T_N<\tau\}\subset\{T_N<K\}\cup\{\tau>K\}.$$
This implies that $$\mathbf{P}_{\mu}(T_N<\tau)\leq \mathbf{P}_{\mu}(T_N<K)+\mathbf{P}_{\mu}(\tau>K).$$
This will prove especially useful whenever $\tau$ is bounded, which happens quite often in application, because there exists then $K$ such that $\mathbf{P}_{\mu}(\tau>K)=0$. On the contrary, when $\tau$ is not bounded, it remains however sometimes possible to obtain $K$ such that $\mathbf{P}_{\mu}(\tau>K)$ be small.
\subsubsection*{Example 1 : a crack propagation model}\label{exemple1}\textit{
We adapt here an example studied by J. Chiquet and N. Limnios in \cite{chiquet}, which models a crack propagation. $Y_t$ is a real-valued process representing the crack size satisfying :
$$\left\{\begin{array}{l}
Y_0>0\\
\dot{Y}_t=A_t Y_t \text{ for all $t \geq 0$.}
\end{array}\right.
$$
where $A_t$ is a Markov process with state space $\{\alpha,\beta\}$ where $0<\alpha \leq \beta$. We are interested in the time $\tau$ before the crack size reaches a critical size $y_c$. $X_{t}=(A_{t},Y_{t})$ is a PDMP, $A_{t}$ representing the mode at time $t$. It is possible to bound the exit time by considering the slowest flow : one clearly has for all $t\geq 0$, $Y_t\geq Y_0 e^{\alpha t}$ and thus $\mathbf{P}_{\mu}(\tau>\frac{1}{\alpha}\ln(\frac{y_c}{Y_0}))=0$.}
\paragraph{}
We now intend to bound $\mathbf{P}_{\mu}(T_N<K)$ for a fixed $K>0$. Let

$$H(t)=
\left\{\begin{array}{ll}
1 & \text{if } t\leq0,\\
e^{-C_\lambda t} & \text{if } 0\leq t<\epsilon, \\
0 & \text{if }t\geq \epsilon.
\end{array}\right.$$

Distribution $H$ represents, roughly speaking, the worst distribution of the inter-jump times $S_k$ in the sense that it is the one that implies the most frequent jumps. Indeed, denote $F_k$ the survival function of $S_k$, one has $H\leq F_k$ for all $k\in \mathbb{N}$. Therefore, Lemma \ref{LemmeSkTilde} provides a random variable $\widetilde{T}_N=\sum_{k=0}^N\widetilde{S}_k$ where $\widetilde{S}_k$ are independent and have survival function $H$ such that $$\mathbf{P}_{\mu}(T_N<K)\leq \mathbf{P}_{\mu}(\widetilde{T}_N<K).$$

\noindent We now bound $\mathbf{P}_{\mu}(\widetilde{T}_N<K)$. Standard computations yield $\E_{\mu}[\widetilde{T}_{N}]=Nm$ and $\mathbf{V}_{\mu}[\widetilde{T}_{N}]=N\sigma^{2}$ where :
\begin{align*}
m:=\E_{\mu}[\widetilde{S}_{1}]    &=  \frac{1}{C_{\lambda}}\left(1-e^{-C_{\lambda}\epsilon}\right), \\
\sigma^{2}:=\mathbf{V}_{\mu}[\widetilde{S}_{1}]    &  = \frac{1}{C_{\lambda}^{2}}\left(1-2C_{\lambda}\epsilon e^{-C_{\lambda}\epsilon}-e^{-2C_{\lambda}\epsilon}\right).
\end{align*}
Assume now that $N$ is such that $Nm>K$ and notice that $$\mathbf{P}_{\mu}(\widetilde{T}_N<K)    \leq  \mathbf{P}_{\mu}\left(\big|\widetilde{T}_N-\E_{\mu}[\widetilde{T}_N]\big|>\E_{\mu}[\widetilde{T}_{N}]-K\right),$$ Tchebychev inequality yields :
\begin{equation*}
\mathbf{P}_{\mu}(\widetilde{T}_N<K)    \leq \frac{N\sigma^{2}}{\left(Nm-K\right)^{2}}.
\end{equation*}
and that the right-hand side term goes to zero when $N$ goes to infinity.
\paragraph{}Finally, when $\tau$ is bounded with a high probability and when Assumptions \ref{hyp-lambda-bornŽ} and \ref{Hyp-Davis-Aepsilon} are fulfilled, we are able to choose $N$ \textit{a priori} such that $\mathbf{P}_{\mu}(T_N<\tau)$ be small. These conditions are satisfied in a large class of applications.

\section{Approximation scheme}\label{section-approximation}

\subsection{The quantization algorithm}
First of all, let us describe the quantization procedure for a random variable and recall some important properties that will be used in the sequel. There exists an extensive literature on quantization methods for random variables and processes. We do not pretend to present here an exhaustive panorama of these methods. However, the interested reader may for instance, consult the following works \cite{quantif1,gray,quantif} and references therein.
Consider $X$ an $\mathbb{R}^q$-valued random variable such that $\big\| X \big\|_p < \infty$ where $\big\| X \big\|_p$ denotes the $L_{p}$-nom of $X$: $\big\| X \big\|_p=\Big( \mathbb{E}[|X|^{p}]\Big)^{1/p}$. 

\bigskip

\noindent
Let $K$ be a fixed integer, the optimal $L_{p}$-quantization of the random variable $X$ consists in finding the best possible $L_{p}$-approximation of $X$ by a random vector $\widehat{X}$ taking at most $K$ values: $\widehat{X}\in \{x^{1},\ldots,x^{K}\}$.
This procedure consists in the following two steps:
\begin{enumerate}
\item Find a finite weighted grid $\Gamma\subset \mathbb{R}^q$ with $\Gamma= \{x^{1},\ldots,x^{K}\}$.
\item Set $\widehat{X}=\widehat{X}^{\Gamma}$ where $\widehat{X}^{\Gamma}=proj_{\Gamma}(X)$ with $p_{\Gamma}$ denotes the closest neighbour projection on $\Gamma$.
\end{enumerate}

\noindent
The asymptotic properties of the $L_{p}$-quantization are given by the following result, see e.g. \cite{quantif}.
\begin{Theorem}
\label{theore}
If $\mathbb{E}[|X|^{p+\eta}]<+\infty$ for some $\eta>0$ then one has
\begin{eqnarray*}
\lim_{K\rightarrow \infty} K^{p/q} \min_{|\Gamma|\leq K} \| X-\widehat{X}^{\Gamma}\|^{p}_{p}& =& J_{p,q} 
\int |h|^{q/(q+p)}(u)du,
\end{eqnarray*}
where the law of $X$ is $P_{X}(du)=h(u) \lambda_{q}(du)+\nu$ with $\nu\perp\lambda_{d}$, $J_{p,d}$ a constant and $ \lambda_{q}$ the Lebesgue measure in $\mathbb{R}^{q}$.
\end{Theorem}
Remark that $X$ needs to have finite moments up to the order $p+\eta$ to ensure the above convergence.
There exists a similar procedure for the optimal quantization of a Markov chain $\{X_{k}\}_{k\in \N}$. There are two approaches to provide the quantized approximation of a Markov chain. The first one, based on the quantization at each time $k$ of the random variable $X_{k}$ is called the \textit{marginal quantization}. The second one that enhances the preservation of the Markov property is called \textit{Markovian quantization}. Remark that for the latter, the quantized Markov process is not homogeneous. These two methods are described in details in \cite[section 3]{quantif}. In this work, we used the marginal quantization approach for simplicity reasons.

Our approximation methods are based on the quantization of the underlying discrete time Markov chain $(\Theta_k)_{k\leq N}=(Z_k,T_k)_{k\leq N}$.
The quantization algorithm provides for each time step $0\leq k\leq N$ a finite grid $\Gamma_k$ of $E\times\R^+$ as well as the transition matrices $(\widehat{Q}_k)_{0\leq k\leq N-1}$ from $\Gamma_k$ to $\Gamma_{k+1}$. Let~$p~\geq~1$ such that for all $k\leq N$, $Z_k$ and $T_k$ have finite moments at least up to order~$p$ and let $proj_{\Gamma_{k}}$ be the nearest-neighbor projection from $E\times\R^+$ onto $\Gamma_k$. The quantized process $(\widehat{\Theta}_k)_{k\leq N}=(\widehat{Z}_k,\widehat{T}_k)_{k\leq N}$ with value for each $k$ in the finite grid $\Gamma_k$ of $E\times\R^+$ is then defined by
\begin{equation}(\widehat{Z}_k,\widehat{T}_k)=proj_{\Gamma_{k}}(Z_k,T_k).\label{Z_n hat Z_n-mesurable}\end{equation}


In practice, we begin with the computation of the quantization grids which merely requires to be able to simulate the process. These grids are only computed once and for all and may be stored off-line. Our schemes are then based on the following simple idea: we replace the process by its quantized approximation within the different recursions. The results are obtained in a very simple way since the quantized process has finite state space.

\begin{Remark}\label{rmq-U-convex} In addition, we recall a technical property of the quantization algorithm proved by C. Bouton and G. Pagès in \cite{BoutonPags} : the quantized process evolves within the convex hull of the support of the law of the original process. Therefore, and it will be required below, Assumption \ref{hyp-U-convex} yields that, if $Z_{k}\in U$ a.s. for some $k\in \{0,...,N\}$ then $\widehat{Z}_{k}\in U$ a.s.
\end{Remark}

\subsection{Approximation scheme of the distribution and proof of convergence}

We already noticed in Proposition \ref{rec_pk} that $p_N(s)=\PP_\mu(\tau>s|\tau\leq T_N)$ may be computed as soon as the sequences $(q_k)_{k\leq N}$ and $(r_k)_{k\leq N-1}$ are known. Therefore, we will find expressions of these sequences depending on the Markov chain $(Z_k,T_k)_{k\leq N}$ that we will replace by the quantized process $(\widehat{Z}_k,\widehat{T}_k)_{k\leq N}$ in order to define their quantized approximations $(\widehat{q}_k)_{k\leq N}$ and $(\widehat{r}_k)_{k\leq N-1}$.\\

First, notice that $\{T_k<\tau\}=\{Z_k\in U\}$ and $\{\tau\leq T_k\}=\{Z_k \not\in U\}$ thanks to Assumption \ref{hypNonRetourU}. Moreover, on $\{Z_k\in U,Z_{k+1} \not\in U\}$, one has $\tau=(T_k+u^*(Z_k))\wedge T_{k+1}$ a.s. where $u^*(x)$ is the deterministic exit time from $U$ starting from the point $x$ (see Definition \ref{Defu*}), and one has :

\begin{equation}\label{Def-qk-rk}
\left\{\begin{array}{ll}
q_k&=\E_{\mu}[\1_{U^c}(Z_k)],\\
r_k(s)&=\E_{\mu}[\1_{\{(T_k+u^*(Z_k))\wedge T_{k+1}>s\}}\1_{U}(Z_k)\1_{U^c}(Z_{k+1})].
\end{array}\right.
\end{equation}
The above equations are crucial in our discussion and, from now on, we will use them without referring to Assumption \ref{hypNonRetourU}.

Before turning to the approximation scheme itself, let us state some properties of the sequence $(q_{k})_{k\leq N}$ that will be important in the following proofs. Indeed, the sequence $(q_{k})_{k}$ increases since $\{\tau\leq T_k\}\subset\{\tau\leq T_{k+1}\}$ for all $k\leq N-1$. Moreover, note that $q_0=0$ and $\lim_{n\rightarrow +\infty}q_{n}=1$ thanks to Assumption \ref{hyp-tau-fini}. Therefore, there exists an index denoted $\tilde{k}\geq 1$ such that
\begin{itemize}
\item{for all $k<\tilde k$, one has $q_{k}=0$,}
\item{for all $k\geq\tilde k$, one has $q_{k}>0$.}
\end{itemize}
We denote $\tilde{q}=q_{\tilde{k}}$ the first positive value of the sequence so that $q_k\geq \tilde{q}$ for all $k\geq\tilde{k}$. One obtains then the following definition.

\begin{Definition}\label{q-tilde} Let 
$$\tilde{k}=\inf\left\{k\geq 0 \text{ such that } q_{k}>0 \right\},$$
$$\tilde{q}=q_{\tilde{k}}$$
i.e. $\tilde{q}$ is the first strictly positive value of the sequence $(q_{k})_{k\in\{0,...,N\}}$.
\end{Definition}

We now naturally define the quantized approximations of the previous sequences.
\begin{Definition}
For all $s>0$, define the sequences $(\widehat{q}_{k})_{k\in\{0,...,N\}}$ and $(\widehat{r}_{k})_{k\in\{0,...,N-1\}}$ by:
\begin{equation*}\left\{
\begin{array}{ll}
\widehat{q}_{k}&=\E_{\mu}[\1_{U^{c}}(\widehat{Z}_k)],\\
\widehat{r}_{k}(s)&=\E_{\mu}[\1_{\{(\widehat{T}_k+u^*(\widehat{Z}_k))\wedge \widehat{T}_{k+1}>s\}}\1_{U}(\widehat{Z}_k)\1_{U^c}(\widehat{Z}_{k+1})].
\end{array}\right.
\end{equation*}
\end{Definition}
It is important to notice that both $\widehat{q}_{k}$ and $\widehat{r}_{k}(s)$ may be computed easily from the quantization algorithm. Indeed, one has :
$$\widehat{q}_{k}=\sum_{\begin{footnotesize}\begin{array}{c}\theta=(z,t)\in\Gamma_k\\ z\not\in U\end{array}\end{footnotesize}}\mathbf{P}(\widehat{\Theta}_k=\theta),$$

$$\widehat{r}_{k}(s)=\sum_{\begin{footnotesize} \begin{array}{c}\theta=(z,t)\in\Gamma_k\\ z\in U\end{array}\end{footnotesize}}\sum_{\begin{footnotesize}\begin{array}{c}\theta'=(z',t')\in\Gamma_{k+1}\\z'\not\in U\end{array}\end{footnotesize}}  \mathbbm{1}_{\begin{large}\{(t+u^*(z))\wedge t'>s\}\end{large}}\mathbf{P}(\widehat{\Theta}_k=\theta)\widehat{Q}_k(\theta;\theta').$$

Recall from Proposition \ref{rec_pk} that the sequence $(p_k)_{k\leq N}$ satisfies a recursion depending on two parameters: $(q_k)_{k\leq N}$ and $(r_k)_{k\leq N-1}$, that we are now able to approximate. Hence, replacing them by their quantized approximations within the same recursion leads to a new sequence denoted $(\widehat{p}_k)_{k\leq N}$. The rest of this section is dedicated to the proof of the convergence of $(\widehat{p}_k)_{k\leq N}$ towards $(p_k)_{k\leq N}$. This convergence is far from being trivial because on the one hand, the definitions of the sequences $(q_k)_{k\leq N}$ and $(r_k)_{k\leq N-1}$ contain many indicator functions that are not Lipschitz continuous and on the other hand, the recursive function giving $p_{k+1}$ from $p_k$, $q_k$, $q_{k+1}$ and $r_k$ is not Lipschitz continuous either.

\begin{Definition}\label{def-rec-hat-pk}
For all $s>0$ and for all $k\in\{0,...,N-1\}$, let $\widehat{p}_{0}(s) =0$ and
\begin{equation}
\widehat{p}_{k+1}(s) =
\left\{\begin{array}{ll}
\frac{\widehat{p}_{k}(s)\widehat{q}_{k}+\widehat{r}_{k}(s)}{\widehat{q}_{k+1}}, &\qquad\text{if $\widehat{q}_{k+1}\neq 0$}\\
0 & \qquad \text{otherwise.}
\end{array}\right.
\end{equation}
\end{Definition}

The two following propositions will be necessary to prove the convergence of the approximation scheme. They respectively state the convergence of $(\widehat{q}_k)_{k\leq N}$ and $(\widehat{r}_k)_{k\leq N-1}$ towards $(q_k)_{k\leq N}$ and $(r_k)_{k\leq N-1}$.

\begin{Proposition}\label{conv_qk}Under Assumptions \ref{hyp-proba-U-apha} and \ref{hypNonRetourU}, for all $k\in\{0,...,N\}$, $\widehat{q}_{k}$ converges towards $q_k$ when the quantization error $\|\Theta_k-\widehat{\Theta}_k\|_p$ goes to zero. More precisely, the error is bounded by
$$|q_{k}-\widehat{q}_{k}|\leq C^{\frac{p}{p+\beta}}\left(\left(\frac{\beta}{p}\right)^{\frac{p}{p+\beta}}+\left(\frac{p}{\beta}\right)^{\frac{\beta}{p+\beta}}\right) \|Z_{k}-\widehat{Z}_{k}\|_{p}^{\frac{p\beta}{p+\beta}},$$
where $C$ and $\beta$ are defined in Assumption \ref{hyp-proba-U-apha}.
\end{Proposition}
\begin{proof} For all $k\in\{0,...,N\}$, equation \eqref{Def-qk-rk} yields
$$|q_{k}-\widehat{q}_{k}|=|\E_{\mu}[\1_{U}(Z_{k})-\1_{U}(\widehat{Z}_{k})]|.$$
The difference of the indicator functions is non zero if and only if $Z_{k}$ and $\widehat{Z}_{k}$ are on either side of $\partial U$. Therefore, in this case, for all $\alpha>0$, if $\big|Z_{k}-\widehat{Z}_{k}\big|\leq\alpha$, then $d(Z_{k},\partial U)\leq \alpha$. Hence, either $\big|Z_{k}-\widehat{Z}_{k}\big|>\alpha$ or $Z_{k}\in U^{\alpha}$. Markov inequality and Assumption \ref{hyp-proba-U-apha} yield~:
\begin{align*}
\E_{\mu}\big|\1_{U}(Z_{k})-\1_{U}(\widehat{Z}_{k})\big|& \leq \PP_{\mu}\big(\big|Z_{k}-\widehat{Z}_{k}\big|>\alpha\big)+
\PP_{\mu}\big(Z_{k}\in U^{\alpha} \big)\\
&\leq  \frac{\|Z_{k}-\widehat{Z}_{k}\|_{p}^{p}}{\alpha^{p}} +C\alpha^{\beta}.
\end{align*}
This bound reaches a minimum when $\alpha=\left(\frac{p\|Z_{k}-\widehat{Z}_{k}\|_{p}^{p}}{\beta C}\right)^{\frac{1}{p+\beta}}$ and the result follows.\hfill
$\Box$
\end{proof}

\begin{Proposition}\label{conv_rk}Under Assumptions \ref{hyp_u*_lip}.a, \ref{hyp-proba-U-apha} and \ref{hypNonRetourU}, for all $k\in\{0,...,N-1\}$ and for almost every $s>0$ w.r.t. the Lebesgue measure on $\R$,
$$\widehat{r}_{k}(s)\rightarrow r_k(s)$$
when the quantization errors $\|\Theta_l-\widehat{\Theta}_l\|_p$ for $l\in \{k,k+1\}$ goes to zero.
\end{Proposition}
\begin{proof} Let $k\in\{0,...,N-1\}$ and $s>0$, equation \eqref{Def-qk-rk} yields
\begin{equation*}
|r_{k}(s)-\widehat{r}_{k}(s)|\leq A+B.
\end{equation*}
where
\begin{align*}
A&=\Big|\mathbf{E}_{\mu}\left[\Big(\mathbbm{1}_{\{(T_k+u^*(Z_k))\wedge T_{k+1}>s\}}-\mathbbm{1}_{\{(\widehat{T}_k+u^*(\widehat{Z}_k))\wedge \widehat{T}_{k+1}>s\}}\Big)\1_{U}(Z_k)\1_{U^c}(Z_{k+1})\right]\Big|,\\
B&=\Big|\mathbf{E}_{\mu}\left[\mathbbm{1}_{\{(\widehat{T}_k+u^*(\widehat{Z}_k))\wedge \widehat{T}_{k+1}>s\}}\Big(\1_{U}(Z_{k})\1_{U^c}(Z_{k+1})-\1_{U}(\widehat{Z}_k)\1_{U^c}(\widehat{Z}_{k+1})\Big)\right]\Big|.
\end{align*}
In the $A$ term, we crudely bound $\1_{U}(Z_k)$ and $\1_{U^c}(Z_{k+1})$ by $1$ and turn to the difference of the two indicator functions. This difference is non zero if and only if $(T_k+u^*(Z_k))\wedge T_{k+1}$ and $(\widehat{T}_k+u^*(\widehat{Z}_k))\wedge \widehat{T}_{k+1}$ are on either side of $s$ yielding that they both belong to $[s-\eta;s+\eta]$ where $\eta=\big|(T_k+u^*(Z_k))\wedge T_{k+1}-(\widehat{T}_k+u^*(\widehat{Z}_k))\wedge \widehat{T}_{k+1}\big|$, one has then:
$$\big|\mathbbm{1}_{\{(T_k+u^*(Z_k))\wedge T_{k+1}>s\}}-\mathbbm{1}_{\{(\widehat{T}_k+u^*(\widehat{Z}_k))\wedge \widehat{T}_{k+1}>s\}}\big|\leq \1_{\{|(T_k+u^*(Z_k))\wedge T_{k+1}-s|\leq \eta\}}$$
so that
$$A\leq \mathbf{P}_{\mu}\left(\big|(T_k+u^*(Z_k))\wedge T_{k+1}-s\big|\leq \eta\right).$$

The following discussion consists in noticing that either $\eta$ is small and so is the probability that $(T_k+u^*(Z_k))\wedge T_{k+1}$ belongs to the interval $[s-\eta;s+\eta]$, or $\eta$ is large but this happens with a small probability too when the quantization error goes to zero. For all $\alpha>0$, one has

\begin{align*}
A&\leq \mathbf{P}_{\mu}\left(\big|(T_k+u^*(Z_k))\wedge T_{k+1}-s\big|\leq \eta, \eta\leq\alpha\right)+\mathbf{P}_{\mu}\left(\eta>\alpha\right)\\
&\leq \mathbf{P}_{\mu}\left(\big|(T_k+u^*(Z_k))\wedge T_{k+1}-s\big|\leq \alpha\right)+\mathbf{P}_{\mu}\left(\eta>\alpha\right)\\
&\leq \big|\varphi_k(s+\alpha)-\varphi_k(s-\alpha)\big| + \frac{\|\eta\|_p^p}{\alpha^p}
\end{align*}
where $\varphi_k$ denotes the distribution function of $(T_k+u^*(Z_k))\wedge T_{k+1}$. Let $\epsilon>0$ and assume that $s$ is not an atom of this distribution so that there exists $\alpha_1>0$ such that $\big|\varphi_k(s+\alpha_1)-\varphi_k(s-\alpha_1)\big|\leq\epsilon$. Besides, thanks to Assumption \ref{hyp_u*_lip}.a stating the Lipschitz continuity of $u^{*}$, one has $\eta\leq\big|T_k-\widehat{T}_k\big|+[u^*]\big|Z_k-\widehat{Z}_k\big|+\big|T_{k+1}-\widehat{T}_{k+1}\big|$. Moreover, since the quantization error goes to $0$, one may assume that $\|\eta\|_p\leq\alpha_1\epsilon^{\frac{1}{p}}$. Setting $\alpha=\alpha_1$ in the previous computations yields

\begin{align*}
A&\leq \big|\varphi_k(s+\alpha_1)-\varphi_k(s-\alpha_1)\big| + \frac{\|\eta\|_p^p}{\alpha_1^p}\leq 2\epsilon.
\end{align*}
Notice that the set of the atoms of the distribution function of $(T_k+u^*(Z_k))\wedge T_{k+1}$ is at most countable so that the previous discussion is true for almost every $s>0$ w.r.t. the Lebesgue measure.
Let us now bound the $B$ term:
\begin{align*}
B&\leq \mathbf{E}_{\mu}\big|\1_{U}(Z_{k})\1_{U^c}(Z_{k+1})-\1_{U}(\widehat{Z}_k)\1_{U^c}(\widehat{Z}_{k+1})\big|\\
&\leq \E_{\mu}\left[\1_{U^c}(Z_{k+1})\big|1_{U}(Z_k)-1_{U}(\widehat{Z}_{k})\big|\right]+\E_{\mu}\left[1_{U}(\widehat{Z}_k)\big|1_{U^c}(Z_{k+1})-1_{U^c}(\widehat{Z}_{k+1})\big|\right]\\
&\leq \big|q_{k}-\widehat{q}_{k}\big|+\big|q_{k+1}-\widehat{q}_{k+1}\big|
\end{align*}
that goes to zero thanks to Proposition \ref{conv_qk}.\hfill
$\Box$\\
\end{proof}

The convergence of the approximation scheme of the distribution of the exit time is now a straightforward consequence of the following proposition.

\begin{Proposition}\label{lemme pi_{k} rho_{k}} We assume Assumptions \ref{hyp-tau-fini}, \ref{hyp-U-convex}, \ref{hyp-proba-U-apha} and \ref{hypNonRetourU} hold. Let $(\sigma_{k})_{k\leq N-1}$ and $(\widehat{\sigma}_{k})_{k\leq N-1}$ be two sequences of $[0,1]$-valued real numbers.
Let $(\pi_{k})_{0\leq k\leq N}$ and $(\widehat{\pi}_{k})_{0\leq k\leq N}$ defined as follows, $\pi_{0} =\widehat{\pi}_{0} =0$ and
$$
\pi_{k+1}=
\left\{\begin{array}{ll}
\frac{\pi_{k}q_{k}+\sigma_{k}}{q_{k+1}}, &\qquad\text{if $q_{k+1}\neq 0$}\\
0 & \qquad \text{otherwise.}
\end{array}\right.$$
$$
\widehat{\pi}_{k+1} =
\left\{\begin{array}{ll}
\frac{\widehat{\pi}_{k}\widehat{q}_{k}+\widehat{\sigma}_{k}}{\widehat{q}_{k+1}}, &\qquad\text{if $\widehat{q}_{k+1}\neq 0$}\\
0 & \qquad \text{otherwise.}
\end{array}\right.$$

\noindent For $0\leq k\leq N$, if the quantization error is such that for all $l\leq k$
$$C^{\frac{p}{p+\beta}}\left(\left(\frac{\beta}{p}\right)^{\frac{p}{p+\beta}}+\left(\frac{p}{\beta}\right)^{\frac{\beta}{p+\beta}}\right) \|Z_{l}-\widehat{Z}_{l}\|_{p}^{\frac{p\beta}{p+\beta}}\leq \frac{1}{2}\tilde{q},$$
then
\begin{eqnarray*}
|\pi_{k}-\widehat{\pi}_{k}|&\leq& \frac{2}{\tilde{q}}\Big(\pi^{sup}|q_{k-1}-\widehat{q}_{k-1}|+|\pi_{k-1}-\widehat{\pi}_{k-1}|+|\sigma_{k-1}-\widehat{\sigma}_{k-1}|\Big)\\
&&+\frac{2(\pi^{sup}+1)}{\tilde{q}^2}|q_{k}-\widehat{q}_{k}|
\end{eqnarray*}
where $\pi^{sup}=\max_{0\leq k\leq N}\pi_{k}$.
\end{Proposition}

\begin{proof} The difficulty of this proof lies in the fact that the recursive function giving $\pi_{k+1}$ from $\pi_k$, $q_k$, $q_{k+1}$ and $\sigma_k$ is not Lipschitz continuous because of the division by $q_{k+1}$. To overcome this drawback, we will use the strictly positive lower bound for $q_{k}$ described earlier. Indeed, recall from Definition \ref{q-tilde} that there exists a step $\tilde{k}$ such that $q_{k}\geq \tilde{q}>0$ for all $k\geq \tilde{k}$ and $q_{k}=0$ for all $k<\tilde{k}$. What is more, a similar bound will be derived for the quantized values $\widehat{q}_{k}$ thanks to the convergence of $\widehat{q}_{k}$ towards $q_{k}$.\\

We now prove by induction that $\widehat{\pi}_k$ converges towards $\pi_k$. First, one has $\widehat{\pi}_0=\pi_0=0$. Then, let $k\in \{1,...,N\}$.\\

If $k<\tilde{k}$, then $q_{k}=0$ and Assumption \ref{hyp-U-convex} yields that $\widehat{q}_k=0$ too. Indeed, $q_k=0$ means that $Z_{k}\in U$ a.s. Since $U$ is a convex set, Remark \ref{rmq-U-convex} implies that $\widehat{Z}_{k}\in U$ a.s. too. In other words, $\widehat{q}_k=0$. Finally, from the definitions, one has $\pi_{k}=\widehat{\pi}_{k}=0$.

If $k\geq\tilde{k}$, then $q_{k}\geq \tilde{q}>0$. In order to bound the error between $\pi_{k}$ and $\widehat{\pi}_{k}$, it is indeed necessary to have a strictly positive lower bound for $q_{k}$ because of the division by $q_{k}$ within the recursion. Now we need to obtain the same kind of bound for $\widehat{q}_{k}$. This can be achieved thanks to Proposition \ref{conv_qk} giving the convergence of $\widehat{q}_{k}$ towards $q_{k}$. Indeed, assume from now on that the number of points in the quantization grids is large enough such that the quantization error is sufficiently small to ensure that for all $j= \tilde{k},...,N$, $|q_j-\widehat{q}_j|\leq \frac{1}{2}\tilde{q}$. Hence, the required lower bound is $\widehat{q}_{k}\geq \frac{1}{2}\tilde{q}>0$. Therefore,
\begin{align*}
|\pi_{k}-\widehat{\pi}_{k}|\leq& \left|\frac{\pi_{k-1}q_{k-1}+\sigma_{k-1}}{q_{k}}-\frac{\widehat{\pi}_{k-1}\widehat{q}_{k-1}+\widehat{\sigma}_{k-1}}{\widehat{q}_{k}}\right|\\
\leq& \frac{\pi_{k-1}}{\widehat{q}_{k}}|q_{k-1}-\widehat{q}_{k-1}|+\frac{\widehat{q}_{k-1}}{\widehat{q}_{k}}|\pi_{k-1}-\widehat{\pi}_{k-1}|+\frac{1}{\widehat{q}_{k}}|\sigma_{k-1}-\widehat{\sigma}_{k-1}|\\
&+|\pi_{k-1}q_{k-1}+\sigma_{k-1}|\frac{|q_{k}-\widehat{q}_{k}|}{q_{k}\widehat{q}_{k}}\\
\leq& \frac{\pi^{sup}}{\widehat{q}_{k}}|q_{k-1}-\widehat{q}_{k-1}|+\frac{1}{\widehat{q}_{k}}|\pi_{k-1}-\widehat{\pi}_{k-1}|+\frac{1}{\widehat{q}_{k}}|\sigma_{k-1}-\widehat{\sigma}_{k-1}|\\
&+(\pi^{sup}+1)\frac{|q_{k}-\widehat{q}_{k}|}{q_{k}\widehat{q}_{k}}\\
\leq& \frac{2}{\tilde{q}}\left(\pi^{sup}|q_{k-1}-\widehat{q}_{k-1}|+|\pi_{k-1}-\widehat{\pi}_{k-1}|+|\sigma_{k-1}-\widehat{\sigma}_{k-1}|\right)\\
&+\frac{2(\pi^{sup}+1)}{\tilde{q}^2}|q_{k}-\widehat{q}_{k}|
\end{align*}
where $\pi^{sup}=\max_{0\leq k\leq N}\pi_{k}$.\hfill
$\Box$
\end{proof}

\begin{Remark}\label{rq-conv-rate-pi_{k}}
Notice that a bound for the rate of convergence of $\widehat{\pi}_{k}$ towards $\pi_{k}$ may be obtained as soon as a bound for the rate of convergence of $\widehat{\sigma}_{k}$ towards $\sigma_{k}$ and an upper bound for the sequence $(\pi_{k})_{0\leq k\leq N}$ are available.
\end{Remark}

Eventually, let us state one of our main results, namely the convergence of the approximation scheme of the distribution of the exit time:
\begin{Theorem}\label{Theo_conv_pk}Under Assumptions \ref{hyp-tau-fini}, \ref{hyp_u*_lip}.a, \ref{hyp-U-convex}, \ref{hyp-proba-U-apha} and \ref{hypNonRetourU}, for all $k\in\{0,...,N\}$ and for almost every $s>0$ w.r.t. the Lebesgue measure on $\R$,
$$\widehat{p}_{k}(s) \rightarrow p_{k}(s)$$
when the quantization errors $\|\Theta_j-\widehat{\Theta}_j\|_p$ for $j\in \{0,...,k\}$ go to zero.
\end{Theorem}

\begin{proof} Let $s>0$ such that $(\widehat{r}_{k}(s))_{k}$ converges towards $(r_{k}(s))_{k}$ and apply Proposition \ref{lemme pi_{k} rho_{k}} with $(\sigma_{k})_{k}=(r_{k}(s))_{k}$ and $(\widehat{\sigma}_{k})_{k}=(\widehat{r}_{k}(s))_{k}$ so that $(\pi_{k})_{k}=(p_{k}(s))_{k}$ and $(\widehat{\pi}_{k})_{k}=(\widehat{p}_{k}(s))_{k}$. Finally, notice that $(p_{k}(s))_{k}$ is bounded by 1.\hfill
$\Box$\\
\end{proof}

\begin{Remark}\label{rq-u*-bornŽ-useless}It may be useful to notice that, although it will be crucial in the moments approximation scheme, the boundedness  condition on $u^{*}$ (Assumption \ref{hyp_u*_lip}.b) was unnecessary in this section. Hence, the distribution approximation can be achieved without this hypothesis.
\end{Remark}

Eventually, we obtain an easily computable approximation of the survival function of the exit time. Let us now consider its moments. Of course they may be derived from the distribution but we present in the following section a method to approximate them directly. An important advantage of this method will be to provide a bound for the rate of convergence.

\subsection{Approximation scheme of the moments and rate of convergence} \label{DefRecQuantif}\label{section-update}

Similarly to the distribution, the moments may be approximated thanks to the quantization of the process $(\Theta_{k})_{k\leq N}=(Z_{k},T_{k})_{k\leq N}$. However, it is important to stress the fact that we will be able to derive a rate of convergence for our approximation scheme. One may notice from Proposition \ref{rec_pkj} that, similarly to the case of the distribution, $p_{N,j}=\E_\mu[\tau^{j}|\tau\leq T_N]$ may be computed as soon as the sequences $(q_k)_{k\leq N}$ and $(r_{k,j})_{k\leq N-1}$ are known. The first one has already been approximated in the previous section but we still need to find an expression of the second one depending on the Markov chain $(Z_k,T_k)_k$ to define its quantized approximation $(\widehat{r}_{k,j})_{k\leq N-1}$. Thanks to Assumption \ref{hypNonRetourU}, the same arguments give

\begin{equation}\label{Def-rkj}
r_{k,j}=\E_{\mu}\left[\Big((T_k+u^*(Z_k))\wedge T_{k+1}\Big)^{j}\1_{U}(Z_k)\1_{U^c}(Z_{k+1})\right].
\end{equation}

So that we may now naturally define the quantized approximation of the sequences $(r_{k,j})_{k\leq N-1}$ and $(p_{k,j})_{k\leq N}$.

\begin{Definition}
For all $j\in \N$, define the sequence $(\widehat{r}_{k,j})_{k\in\{0,...,N-1\}}$:
\begin{equation*}
\widehat{r}_{k,j}=\E_{\mu}\left[\Big((\widehat{T}_k+u^*(\widehat{Z}_k))\wedge \widehat{T}_{k+1}\Big)^{j}\1_{U}(\widehat{Z}_k)\1_{U^c}(\widehat{Z}_{k+1})\right].
\end{equation*}
and the sequence $(\widehat{p}_{k,j})_{k\in\{0,...,N\}}$ by $\widehat{p}_{0,j}=0$ and
\begin{equation*}
\widehat{p}_{k+1,j} =
\left\{\begin{array}{ll}
\frac{\widehat{p}_{k,j}\widehat{q}_{k}+\widehat{r}_{k,j}}{\widehat{q}_{k+1}}, &\qquad\text{if $\widehat{q}_{k+1}\neq 0$}\\
0 & \qquad \text{otherwise.}
\end{array}\right.
\end{equation*}
\end{Definition}
As for $\widehat{q}_{k}$ and $\widehat{r}_{k}(s)$ defined in the previous section, $\widehat{r}_{k,j}$ may be computed easily from the quantization algorithm. Indeed, one has :

$$\widehat{r}_{k,j}=\sum_{\begin{footnotesize}\begin{array}{c}\theta=(z,t)\in\Gamma_k\\ z\in U\end{array}\end{footnotesize}}\sum_{\begin{footnotesize}\begin{array}{c}\theta'=(z',t')\in\Gamma_{k+1}\\z'\not\in U\end{array}\end{footnotesize}}  \Big((t+u^*(z))\wedge t'\Big)^{j}\mathbf{P}(\widehat{\Theta}_k=\theta)\widehat{Q}_k(\theta;\theta').$$

The following proposition proves the convergence of $\widehat{r}_{k,j}$ towards $r_{k,j}$.

\begin{Proposition}\label{conv_rkj}Under Assumptions, \ref{hyp_u*_lip}.a, \ref{hyp-proba-U-apha}, \ref{hypNonRetourU} and \ref{hyp_t*_bnd}, for all $k\in\{0,...,N-1\}$ and for all $j\in\N$, $\widehat{r}_{k,j}$ converges towards $r_{k,j}$
when the quantization errors $\|\Theta_l-\widehat{\Theta}_l\|_p$ for $l\in\{k,k+1\}$ go to zero. More precisely, the error is bounded by
\begin{align*}
|r_{k,j} - \widehat{r}_{k,j}|\leq& \text{  } j\big((k+1)C_{t^{*}}\big)^{j-1} \big(\|T_k-\widehat{T}_k\|_{p}+[u^*]\|Z_k-\widehat{Z}_k\|_{p}+\|T_{k+1}-\widehat{T}_{k+1}\|_{p}\big)\\
&+\big((k+1)C_{t^{*}}\big)^{j} \big(|q_{k}-\widehat{q}_{k}|+|q_{k+1}-\widehat{q}_{k+1}|\big).
\end{align*}
\end{Proposition}
\begin{proof} Let $k\in\{0,...,N-1\}$ and $j\in\N$, one has :
\begin{equation*}
|r_{k,j}-\widehat{r}_{k,j}|\leq A+B.
\end{equation*}
where
\begin{align*}
A&=\Big|\mathbf{E}_{\mu}\left[\Big(\big((T_k+u^*(Z_k))\wedge T_{k+1}\big)^{j}-\big((\widehat{T}_k+u^*(\widehat{Z}_k))\wedge \widehat{T}_{k+1}\big)^{j}\Big)\1_{U}(Z_k)\1_{U^c}(Z_{k+1})\right]\Big|,\\
B&=\Big|\mathbf{E}_{\mu}\left[\Big((\widehat{T}_k+u^*(\widehat{Z}_k))\wedge \widehat{T}_{k+1}\Big)^{j}\Big(\1_{U}(Z_{k})\1_{U^c}(Z_{k+1})-\1_{U}(\widehat{Z}_k)\1_{U^c}(\widehat{Z}_{k+1})\Big)\right]\Big|.
\end{align*}
Assumption \ref{hyp_t*_bnd} yields that the inter jump times $S_{i}$ are a.s. bounded by $C_{t^{*}}$ so that $T_{i}\leq i C_{t^{*}}$ a.s. and $(T_i+u^*(Z_i))\wedge T_{i+1}\leq (i+1)C_{t^{*}}$ a.s.. By using Remark \ref{rmq-U-convex}, these bounds are equally true for the quantized process $\widehat{T}_{i}\leq i C_{t^{*}}$ and $(\widehat{T}_i+u^*(\widehat{Z}_i))\wedge \widehat{T}_{i+1}\leq \widehat{T}_{i+1}\leq (i+1)C_{t^{*}}$ a.s.\\
Let us first consider the term $A$, we crudely bound the indicator functions by 1. Moreover, denote $\eta=\big|(T_k+u^*(Z_k))\wedge T_{k+1}-(\widehat{T}_k+u^*(\widehat{Z}_k))\wedge \widehat{T}_{k+1}\big|$ and notice that the function $x\rightarrow x^{j}$ is Lipschitz continuous on any set $[0,M]$ with Lipschitz constant~$jM^{j-1}$.
\begin{align*}
A&\leq \mathbf{E}_{\mu}\left[j\big((k+1)C_{t^{*}}\big)^{j-1}\eta\right]\\
&\leq j\big((k+1)C_{t^{*}}\big)^{j-1} \|\eta\|_{p}
\end{align*}
and thanks to Assumption \ref{hyp_u*_lip}.a stating the Lipschitz continuity of $u^{*}$, one has
$$A\leq j\big((k+1)C_{t^{*}}\big)^{j-1}\Big(\|T_k-\widehat{T}_k\|_{p}+[u^*]\|Z_k-\widehat{Z}_k\|_{p}+\|T_{k+1}-\widehat{T}_{k+1}\|_{p}\Big).$$
Moreover, the term $B$ is bounded by:
\begin{align*}
B&\leq \big((k+1)C_{t^{*}}\big)^{j}\mathbf{E}_{\mu}\left|\1_{U}(Z_{k})\1_{U^c}(Z_{k+1})-\1_{U}(\widehat{Z}_k)\1_{U^c}(\widehat{Z}_{k+1})\right|\\
&\leq \big((k+1)C_{t^{*}}\big)^{j} \big(|q_{k}-\widehat{q}_{k}|+|q_{k+1}-\widehat{q}_{k+1}|\big).
\end{align*}
We conclude thanks to Proposition \ref{conv_qk}.\hfill
$\Box$\\
\end{proof}

We may now state the other important results of our paper namely the convergence of the approximation scheme of the moments of the exit time with a bound for the rate of convergence.
\begin{Theorem}\label{Theo_conv_pkj}Under Assumptions \ref{hyp-tau-fini}, \ref{hyp_u*_lip}.a, \ref{hyp-U-convex}, \ref{hyp-proba-U-apha}, \ref{hypNonRetourU} and \ref{hyp_t*_bnd}, for all $k\in\{0,...,N\}$ and for all $j\in\N$, $\widehat{p}_{k,j}$ converges towards $p_{k,j}$
when the quantization errors $\|\Theta_j-\widehat{\Theta}_j\|_p$ for $j\in \{0,...,k\}$ go to zero. \\
More precisely, if the quantization error is such that for all $l\leq k$
$$C^{\frac{p}{p+q}}\left(\left(\frac{q}{p}\right)^{\frac{p}{p+q}}+\left(\frac{p}{q}\right)^{\frac{q}{p+q}}\right) \|Z_{l}-\widehat{Z}_{l}\|_{p}^{\frac{pq}{p+q}}\leq \frac{1}{2}\tilde{q},$$
then
\begin{eqnarray*}
|p_{k,j}-\widehat{p}_{k,j}|&\leq& \frac{2}{\tilde{q}}\left((NC_{t^{*}})^{j}|q_{k-1}-\widehat{q}_{k-1}|+|p_{k-1,j}-\widehat{p}_{k-1,j}|+|r_{k-1,j}-\widehat{r}_{k-1,j}|\right)\\
&&+\frac{2((NC_{t^{*}})^{j}+1)}{\tilde{q}^2}|q_{k}-\widehat{q}_{k}|.
\end{eqnarray*}
\end{Theorem}

\begin{Remark}The rate of convergence depends on the quantity $\tilde{q}$ whose exact value might be unknown in some complex applications. In that case, it may still be approximated through Monte-Carlo simulations (see examples in Section \ref{section-examples}). Nevertheless, Theorems \ref{Theo_conv_pk} and \ref{Theo_conv_pkj} prove the convergence of our approximation schemes regardless of the value of $\tilde{q}$.
\end{Remark}

\begin{proof} Let $j\in\N$ and apply Proposition \ref{lemme pi_{k} rho_{k}} with $(\sigma_{k})_{k}=(r_{k,j})_{k}$ and $(\widehat{\sigma}_{k})_{k}=(\widehat{r}_{k,j})_{k}$ such that $(\pi_{k})_{k}=(p_{k,j})_{k}$ and $(\widehat{\pi}_{k})_{k}=(\widehat{p}_{k,j})_{k}$. Finally, according to Remark \ref{rq-conv-rate-pi_{k}}, a bound for the rate of convergence is obtained since the sequence $(p_{k,j})_{0 \leq k \leq N}$ is bounded by:
\begin{equation*}
p_{k,j}=\mathbf{E}_{\mu}\left[\tau^{j}\big|\tau\leq T_k\right]
\leq\mathbf{E}_{\mu}\left[T_{k}^{j}\big|\tau\leq T_k\right]
\leq\mathbf{E}_{\mu}\left[(kC_{t^{*}})^{j}\big|\tau\leq T_k\right]
\leq (kC_{t^{*}})^{j}
\leq (NC_{t^{*}})^{j}.
\end{equation*}
Hence, the result. \hfill$\Box$
\end{proof}

\section{Examples and numerical results}\label{section-examples}

\subsection{A Poisson process}\label{exemple-Poisson}

Let $N_t$ be a Poisson process with parameter~$\lambda=1$ and let~$Y_t=t+N_t$. $(Y_t)_{t\geq 0}$ is a PDMP with state space $E=\mathbb{R}$~; inter jump times~$S_k$ have independent exponential distribution with parameter $\lambda=1$ ; the flow is defined on~$(\mathbb{R}^+)^2$ by $\Phi(x,t)=x+t$~; and finally, the post-jump locations satisfy : $\forall x\in E$, $Q(\{x+1\},x)=1$. An example of trajectory of the process is represented in figure \ref{trajectory poisson}. We are interested in the exit time problem for the process~$(Y_t)_{t\geq 0}$. The study of this process is especially interesting because it is possible to compute the exact value of its distribution function in order to compare it with the numerical value given by our approximation scheme.\\

\begin{figure}[h]
\begin{center}
\includegraphics[scale=.5]{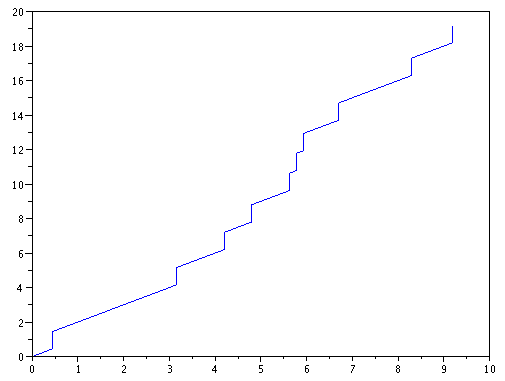}
\caption{A trajectory of the process $(Y_t)$ drawn until the $10^{th}$ jump time.}
\label{trajectory poisson}
\end{center}
\end{figure}

Let us turn now to the numerical simulations. Let $b=10$ i.e. $U=]-\infty,10[$. We may choose $N=10$ since $Y_{T_N}=T_N+N_{T_N}=T_N+N\geq N$. Besides, it is clear that for all $y\in ]-\infty,10[$, $u^{*}(y)=10-y$. Assumptions \ref{hyp_u*_lip} and \ref{hyp-U-convex} are clearly satisfied and so is Assumption \ref{hyp-proba-U-apha} thanks to the following lemma. 
\begin{Lemma}For all $\alpha>0$ and for all $k\in \{0,...,N\}$,
$$\PP_{\mu}\left(Z_{k}\in U^{\alpha}\right)\leq 2\alpha.$$
\end{Lemma}
\begin{proof} Since $Z_{0}=0$ a.s., $\PP_{\mu}\left(Z_{0}\in U^{\alpha}\right)=\PP_{\mu}\left(Z_{0}\in [10-\alpha,10+\alpha]\right)=\1_{\{\alpha\geq 10\}}\leq \frac{1}{10}\alpha\leq 2\alpha$. \\
Let now $k\in \{1,...,N\}$. Denote $f_{\gamma(k,1)}$ the density of the distribution $\gamma(k,1)$ and denote $C_{k}=\frac{1}{(k-1)!}\left(\frac{k-1}{e}\right)^{k-1}$ its bound. Since $T_{k}$ has distribution $\gamma(k,1)$, $Z_{k}=k+T_{k}$ has density $f_{Z_{k}}(\cdot)=f_{\gamma(k,1)}(\cdot-k)$ that is also bounded by $C_{k}$. Eventually, one has:
\begin{equation*}
\PP_{\mu}\left(Z_{k}\in U^{\alpha}\right)=\PP_{\mu}\left(Z_{k}\in [10-\alpha,10+\alpha]\right)\leq 2C_{k}\alpha\leq 2\alpha.
\end{equation*}
Indeed, the sequence $(C_{k})_{k}$ decreases so that for all $k\in \{1,...,N\}$, $C_{k}\leq C_{1}=1.$
\hfill$\Box$
\end{proof}

Moreover, Assumption \ref{hypNonRetourU} is satisfied since the process increases but Assumption \ref{hyp_t*_bnd} is not, because $t^{*}(x)=+\infty$ for all $x\in E$. However, as pointed out in Section \ref{section-exit-time}, this can be solved by considering the process killed at time $\tau$.

\subsubsection*{The mean exit time}

Table \ref{poisson-moment1} presents the simulations results for the approximation of the mean exit time. It includes for different number of points in the quantization grids the value of $\widehat{p}_{N,1}$ which approximates the mean exit time. A reference value is obtained thanks to Monte Carlo method ($10^{6}$ simulations): $\E[\tau_{10}]_{\text{Monte Carlo}}=5.125$.

\begin{table}[h]
\begin{center}
\begin{tabular}{|c||c|c|}
\hline
Points in the quantization grids &$\widehat{p}_{N,1}$ & relative error to $5.125$\\
\hline
$20$ points&5.050&1.46 \% \\
\hline
$50$ points& $5.096$ & 0.56 \%\\
\hline
$100$ points& $5.095$ & 0.58 \%\\
\hline
$200$ points& $5.118$ & 0.13 \% \\
\hline
$300$ points& 5.128& 0.06 \%\\
\hline
$500$ points& 5.123& 0.03 \%\\
\hline
\end{tabular}
\end{center}
\caption{Simulations results for the mean exit time}
\label{poisson-moment1}
\end{table}

\subsubsection*{The second moment}

We present the results of the approximation of the second moment in Table \ref{poisson-moment2}. Our Monte Carol reference value ($10^{6}$ simulations) is $\E[\tau_{10}^{2}]_{\text{Monte Carlo}}=27.5$.

\begin{table}[h]
\begin{center}
\begin{tabular}{|c||c|c|}
\hline
Points in the quantization grids &$\widehat{p}_{N,2}$ & relative error to $27.5$\\
\hline
$20$ points&26.66&3.05 \% \\
\hline
$50$ points& $27.20$ & 1.11 \%\\
\hline
$100$ points& $27.21$ & 1.05 \%\\
\hline
$200$ points& $27.43$ & 0.25 \% \\
\hline
$300$ points& 27.54& 0.13 \%\\
\hline
$500$ points& 27.49& 0.03 \%\\
\hline
\end{tabular}
\end{center}
\caption{Simulations results for the second moment}
\label{poisson-moment2}
\end{table}

For the first and second moment, the empirical convergence rate is presented on Figure \ref{rate-moment-poisson}.
It is estimated through a regression model as $-1.23$ for the first moment and $-1.39$ for the second moment. Remark that
there are roughly the same order as the rate of convergence of the optimal quantizer (see Theorem \ref{theore})  as here the dimension is $1$.
\begin{figure}[hbtp]
\begin{center}
\includegraphics[scale=.4]{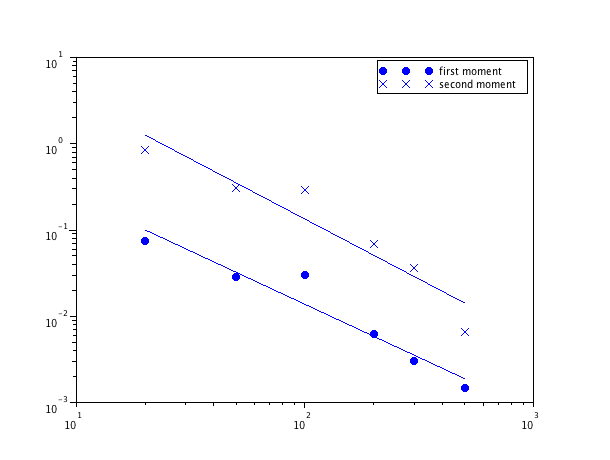}
\caption{Logarithm of the error w.r.t. the logarithm of the number of points in the quantization grids for the first and second moment of the Poisson process.}
\label{rate-moment-poisson}
\end{center}
\end{figure}

\subsubsection*{The exit time distribution}

As announced earlier, one can obtain the exact value of the survival function of the exit time.

\begin{Proposition}\label{prop-distri-espadon}
Denote $fl(.)$ the floor function. For all $s,b\in \R^{+}$, one has :
$$\PP[\tau_{b}\geq s]=
\left\{
\begin{array}{ ll}
\PP[T_{fl(b-s)+1}> s]      & \text{ for all $s\leq b$,}  \\
0 & \text{ otherwise.}
\end{array}\right.
$$
\end{Proposition}
\begin{Remark}
Notice that $T_{k}$ has distribution $\gamma(k,1)$ so that the right-hand side term in the above proposition can be computed easily.
\end{Remark}
\begin{proof}
Let $s>0$. Notice that $Y_{s}\geq s$, thus $\tau_{b}< s$ a.s. when $s>b$. Assume now that $s\leq b$, one has :
\begin{equation*}
\PP[\tau_{b}\geq s]=\PP[Y_{s}\leq b]=\PP[N_{s}\leq b-s]=\PP[N_{s}\leq fl(b-s)]=\PP[T_{fl(b-s)+1}\geq s].
\end{equation*}
Hence, the result. \hfill$\Box$
\end{proof}

Figure \ref{survival function poisson} represents both the exact survival function of the exit time and its quantized approximation. Table \ref{poisson-distrib} contains the empirical error between the two functions.
For the survival function, the empirical convergence rate is presented on Figure \ref{rate-survival-poisson}.
It is estimated through a regression model as $-1.05$. Remark that
it is roughly the same order as the rate of convergence of the optimal quantizer (see Theorem \ref{theore})  as here the dimension is $1$.

\begin{figure}[htbp]
\begin{center}
\includegraphics[scale=.5]{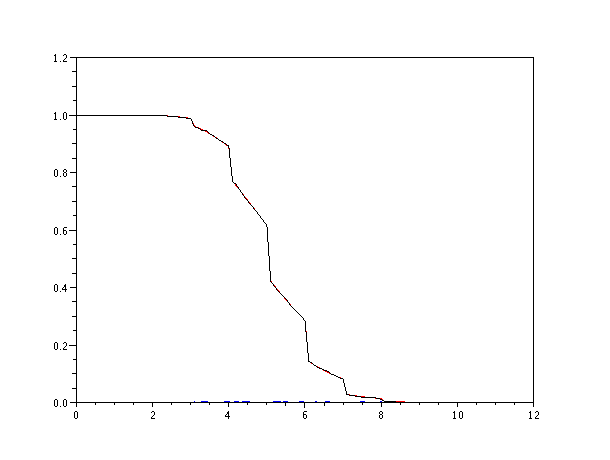}
\caption{Survival function of $\tau_{10}$ and its quantized approximation with 500 points in the quantization grids. The functions appear indistinguishable.}
\label{survival function poisson}
\end{center}
\end{figure}

\begin{table}[h]
\begin{center}
\begin{tabular}{|c||c|}
\hline
Points in the quantization grids &$\max_{s}|p_{N}(s)-\widehat{p}_{N}(s)|$\\
\hline
$20$ points&0.090\\
\hline
$50$ points&0.077\\
\hline
$100$ points&0.057\\
\hline
$200$ points &0.011\\
\hline
$300$ points &0.007\\
\hline
$500$ points &0.005\\
\hline
\end{tabular}
\end{center}
\caption{Simulations results for the distribution}
\label{poisson-distrib}
\end{table}

\begin{figure}[htbp]
\begin{center}
\includegraphics[scale=.4]{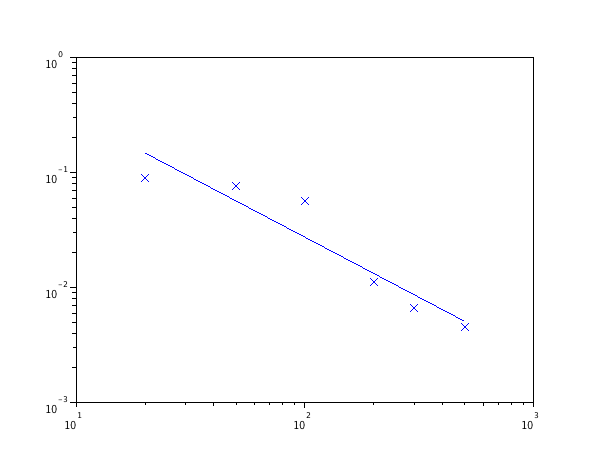}
\caption{Logarithm of the error w.r.t. the logarithm of the number of points in the quantization grids for the survival function of the 
Poisson process.}
\label{rate-survival-poisson}
\end{center}
\end{figure}

\begin{Remark}
We already insisted on the fact that our approach is flexible w.r.t. $U$. In this example, one could obtain very quickly the mean exit time or the exit time distribution for a different set $U'=]-\infty,b']$ for any $0<b'\leq b=10$. Indeed, $\PP(\tau_{b'}>T_{10})=0$ so that it is not necessary to compute new quantization grids.
\end{Remark}

\begin{Remark}
Recall that the value of $T_{k}$ may be obtained from $Z_{k}$ since $T_{k}=Z_{k}-k$ so that it is sufficient to quantize the process $(Z_{k})_{k\leq N}$ instead of $(Z_{k},T_{k})_{k\leq N}$. The reduction of the dimension of the process that has to be quantized results in an improvement of the convergence rate and it appears that the approximations presented in the previous tables converge indeed very quickly.
\end{Remark}

\subsubsection*{Convergence rate for the exit time distribution}

One may notice from the proof of Proposition \ref{conv_rk} that a bound for the rate of convergence for the exit time distribution may be obtained as soon as for all $k\in \{0,...,N-1\}$, the survival function of $(T_{k}+u^{*}(Z_{k}))\wedge T_{k+1}$ denoted $\varphi_{k}$ is piecewise Lipschitz continuous. Although it is tough to state general assumptions under which this is true, the following proposition proves that the condition is fulfilled in our example.

\begin{Proposition} For all $k\in \{0,...,N-1\}$, the survival function $\varphi_{k}$ of $(T_{k}+u^{*}(Z_{k}))\wedge T_{k+1}$ is Lipschitz continuous on $]-\infty;b-k[$ and on $]b-k;+\infty[$ with Lipschitz constant $[\varphi_{k}]\leq 1$.
\end{Proposition}
\begin{proof} 
Let $k=0$ and $s>0$, one has:
\begin{eqnarray*}
\varphi_{0}(s)&=&\PP_{\mu}((T_{0}+u^{*}(Z_{0}))\wedge T_{1}>s)\\
&=&\PP_{\mu}(b \wedge T_{1}>s)\\
&=&\1_{\{b>s\}}\PP_{\mu}(T_{1}>s)\\
&=&\1_{\{b>s\}}e^{- s} \qquad \text{ since $T_{1}$ has exponential distribution with parameter $1$.}
\end{eqnarray*}
Therefore, the function $\varphi_{0}$ is worth zero on $[b;+\infty[$ and is Lipschitz continuous with Lipschitz constant 1 on $]0;b[$.

Let $k\geq 1$, $s>0$ and remember that the random variables $(S_{j})_{j\geq 0}$ are independent and all have exponential distribution with parameter $1$ so that, in particular, $T_{k}$ and $S_{k+1}$ are independent and $T_{k}$ has distribution $\gamma(k,1)$. Moreover, recall that $Z_{k}=k+T_{k}$ and that $u^{*}(x)=b-x$.
\begin{align*}
\varphi_{k}(s)&=\PP_{\mu}((T_{k}+u^{*}(Z_{k}))\wedge T_{k+1}>s)\\
&=\int_{(\R^{+})^{2}}\1_{\{(t+(b-k-t))\wedge u>s\}}f_{\gamma(k,1)}(t)f_{\gamma(k+1,1)}(u)dtdu
\end{align*}
where $f_{\gamma(j,1)}$ denotes the density function of the distribution $\gamma(j,1)$ for $j\in\{k,k+1\}$.

\noindent Let $s'>s>0$, one has:
\begin{align*}
|\varphi_{k}(s')-\varphi_{k}(s)|&\leq\int_{(\R^{+})^{2}}\left|\1_{\{(b-k)\wedge u>s'\}}-\1_{\{(b-k)\wedge u>s\}}\right|f_{\gamma(k,1)}(t)f_{\gamma(k+1,1)}(u)dtdu\\
&\leq\int_{(\R^{+})^{2}}\1_{\{(b-k)\wedge u\in]s;s']\}}f_{\gamma(k,1)}(t)f_{\gamma(k+1,1)}(u)dtdu\\
&\leq\int_{(\R^{+})^{2}}\left(\1_{\{b-k\in]s;s']\}}+\1_{\{u\in]s;s']\}}\right)f_{\gamma(k,1)}(t)f_{\gamma(k+1,1)}(u)dtdu\\
&\leq \1_{\{b-k\in[s;s']\}}+C_{f_{\gamma(k+1,1)}}|s'-s|\\
&\leq \1_{\{b-k\in[s;s']\}}+|s'-s|\qquad\text{since $C_{f_{\gamma(k+1,1)}}=\frac{1}{(k)!}\left(\frac{k}{e}\right)^{k}\leq 1$.}
\end{align*}
Eventually if $s$ and $s'$ both belong to $]0;b-k[$ or if they both belong to $]b-k;+\infty[$, one has $|\varphi_{k}(s')-\varphi_{k}(s)|\leq |s'-s|$. The result follows.
\hfill $\Box$\\
\end{proof}

Consequently, in this example, we are now able to state a bound for the rate of convergence of the exit time distribution approximation scheme. The following proposition is therefore an improvement over Proposition \ref{conv_rk} and Theorem \ref{Theo_conv_pk}.
\begin{Proposition}For all $k\in\{0,...,N-1\}$, let $s>0$ and assume that the quantization error be small enough to ensure that $$\left(\frac{p}{2}\right)^{\frac{1}{p+1}}\Big(\|T_k-\widehat{T}_k\|_{p}+\|Z_k-\widehat{Z}_k\|_{p}+\|T_{k+1}-\widehat{T}_{k+1}\|_{p}\Big)^{\frac{p}{p+1}}<|b-k-s|,$$ 
one has then
\begin{align*}
|r_k(s)-\widehat{r}_{k}(s)|\leq &\text{ }2\left(\frac{p}{2}\right)^{\frac{1}{p+1}}\left(\frac{1}{p}+1\right) \Big(\|T_k-\widehat{T}_k\|_{p}+\|Z_k-\widehat{Z}_k\|_{p}+\|T_{k+1}-\widehat{T}_{k+1}\|_{p}\Big)^{\frac{p}{p+1}}\\
&+\big|q_{k}-\widehat{q}_{k}\big|+\big|q_{k+1}-\widehat{q}_{k+1}\big|.
\end{align*}
Moreover, for all $k\in\{0,...,N\}$, if the quantization error is such that for all $l\leq k$
$$2\left(\frac{p}{2}\right)^{\frac{1}{p+1}}\left(\frac{1}{p}+1\right) \|Z_{l}-\widehat{Z}_{l}\|_{p}^{\frac{p}{p+1}}\leq \frac{1}{2}\tilde{q},$$
one has then
\begin{align*}
|p_{k}(s)-\widehat{p}_{k}(s)|\leq &\frac{2}{\tilde{q}}\Big(|q_{k-1}-\widehat{q}_{k-1}|+|p_{k-1}(s)-\widehat{p}_{k-1}(s)|+|r_{k-1}(s)-\widehat{r}_{k-1}(s)|\Big)\\
&+\frac{4}{\tilde{q}^2}|q_{k}-\widehat{q}_{k}|.
\end{align*}
\end{Proposition}
\begin{proof} The proof derives directly from the proofs of Proposition \ref{conv_rk} and Theorem \ref{Theo_conv_pk}. Simply notice that the $A$ term may be bounded thanks to the piecewise Lipschitz continuity of the functions $\varphi_{k}$ on $]-\infty;b-k[$ and on $]b-k;+\infty[$. Let $s>0$, $s\neq b-k$, and let $\alpha>0$ such that $b-k\not\in[s-\alpha;s+\alpha]$ i.e. $\alpha<|b-k-s|$, one has
\begin{align*}
A&\leq \big|\varphi_k(s+\alpha)-\varphi_k(s-\alpha)\big| + \frac{\|\eta\|_p^p}{\alpha^p}\text{ from the proof of Proposition \ref{conv_rk}}\\
&\leq 2[\varphi_k]\alpha + \frac{\|\eta\|_p^p}{\alpha^p}
\end{align*}
that reaches a minimum when $\alpha=\left(\frac{p\|\eta\|_p^p}{2[\varphi_k]}\right)^{\frac{1}{p+1}}.$ Notice besides that $[\varphi_k]=1$ and $[u^*]=1$. 
\hfill$\Box$
\end{proof}

\begin{Remark}
We can calculate the exact value of $\tilde{q}$ that is the first nonnegative value of the sequence $\big(\PP_{\mu}(Z_{k}\not\in U)\big)_{k}$. One has $\tilde{q}=\PP_{\mu}(Z_{1}\not\in ]-\infty;10[)=\PP_{\mu}(T_{1}\geq 9)=e^{-9}$ because $T_{1}$ has exponential distribution with parameter $1$.
\end{Remark}

\subsection{A corrosion model}

Let us consider a structure of aluminium corroded successively into 3 different environments. Corrosion is prevented by some protection until a random time $\gamma$ when corrosion starts. Then, in each environment $i\in\{1;2;3\}$, the loss of thickness satisfies:
$$d_i(t)=\rho_i\left(t-\gamma+\eta_i\left(e^{-\frac{t-\gamma}{\eta_i}}-1\right)\right)\mathbbm{1}_{\{t\geq\gamma\}}$$
where $\rho_i$ is the corrosion rate ($\rho_{i}$ has a uniform distribution on an interval that depends on the environment $i$) and $\eta_i$ is a constant transition time. The structure goes from environment 1 to environment 2, then from 2 to 3, from 3 to 1 and so on. It remains in environment $i$ for a time $T_i$ which has exponential distribution with parameter $\lambda_i$. When the loss of thickness reaches 0.2 mm, the piece is said to be unusable, this will be the exit criterion. Table \ref{parametres cor} gives the values of the different parameters.

\begin{table}[h]
\begin{center}
\begin{tabular}{|cc||c|c|c|}
\cline{3-5}
\multicolumn{1}{c}{}&&\text{environment 1}&\text{environment 2}&\text{environment 3}\\
\hline
$\lambda_i$ &($\text{h}^{-1}$)& $(17520)^{-1}$ & $(131400)^{-1}$ & $(8760)^{-1}$\\
\hline
$\eta_i$ &(h)& 30000 & 200000& 40000 \\
\hline
$\rho_i$ &(mm.$\text{h}^{-1}$)& $[10^{-6},10^{-5}]$ &  $[10^{-7},10^{-6}]$ &  $[10^{-6},10^{-5}]$ \\
\hline
$\gamma$ &(h)& \multicolumn{3}{|c|}{Weibull distribution with $\alpha=2.5$ and $\beta=11800$} \\
\hline
\end{tabular}
\end{center}
\caption{Numerical values of the parameters of the corrosion model}
\label{parametres cor}
\end{table}

\paragraph{}The loss of thickness will be represented by a PDMP whose modes are the different environments. Let then $M=\{(i,j)\text{ : } i\in\{1,2,3\},j\in \{0,1\}\}$. For $m=(i,j)\in M$, $i$ represents the environment and $j$ is worth $1$ if the protection $\gamma$ is still active and 0 otherwise. For each $m\in M$, let $E_m=\mathbb{R}^4$ and for $\xi\in E_m$, $\xi$ represents the family $(d, s, \rho, \gamma)$ where $d$ is the corroded thickness and $s$ is the time since the last jump. The set $U_{m}$ will therefore be for all $m\in M$, $U_{m}=]-\infty;0.2]\times\mathbb{R}^3$. This set is convex so that Assumption \ref{hyp-U-convex} is satisfied. Finally, the flow in mode $m=(i,j)$ is :
\begin{align*}
\Phi_{(i,0)}(\left(\begin{array}{ccc}d\\ s\\ \rho \\ 0\end{array}\right),t)&=\left(\begin{array}{ccc}d + d_m(t+s) - d_m(s)\\ t+s \\ \rho\\ 0\end{array}\right),\\
\Phi_{(i,1)}(\left(\begin{array}{ccc}0\\ s\\ \rho \\ \gamma\end{array}\right),t)&=\left(\begin{array}{ccc}0\\ t+s \\ \rho\\ (\gamma-t)\mathbbm{1}_{\{\gamma\geq t\}}\end{array}\right).
\end{align*}
The parameters $d$ and $\gamma$ evolve continuously between the jumps but $\rho$ is chosen independently after each jump and is constant along the flow.\\

Let us consider the approximation of the distribution and of the mean exit time. Concerning the first moment, one may notice that $\E_{\mu}[\tau]=\E_{\mu}[\gamma]+\E_{\mu}[\tau']$ where $\gamma$ has Weibull distribution and $\tau'$ represents the exit time in the case of a process without initial protection against corrosion (i.e. $\gamma=0$). Therefore, it is sufficient to check whether $\tau'$ satisfies the required assumptions. Hence, let $\gamma=0$ and notice that $u^{*}$ is then bounded since $\rho\geq 10^{-7}$ and $\eta\leq 200000$ so that $d_{m}(t)\geq 10^{-7}(t-200000)$ and eventually $u^{*}\leq 0.2\times 10^{7}+200000=2.2\times 10^{6}$ h. Denote by $C_{u^{*}}$ this bound. Concerning the distribution, Assumption \ref{hyp_u*_lip}.b (the boundedness condition on $u^{*}$) is not required according to Remark \ref{rq-u*-bornŽ-useless}. Moreover, it is easy to notice from the proofs of Propositions \ref{conv_rk} and \ref{conv_rkj} that Assumption \ref{hyp_u*_lip}.a (the Lipschitz continuity condition on $u^{*}$) becomes useless in this example thanks to Lemma \ref{prop-u*-lip-cor}. Assumption \ref{hyp-proba-U-apha} follows from Lemma \ref{prop-proba-U-alpha-cor} below. Eventually, Assumption \ref{hypNonRetourU} is satisfied but Assumption \ref{hyp_t*_bnd} is not. However, considering the process killed at time $\tau$ solves this issue.

\begin{Lemma}\label{prop-proba-U-alpha-cor}
For all $\alpha>0$ and for all $k\in \{0,...,N\}$,
$$\PP_{\mu}(Z_{k}\in U^{\alpha})\leq 5\alpha.$$
\end{Lemma}
\begin{proof} For notational convenience, introduce $M_{k}$, $D_{k}$, $R_{k}$ and $G_{k}$ the values of $m$, $d$, $\rho$ and $\gamma$ after the $k$-th jump so that $Z_{k}=(M_{k},D_{k},R_{k},G_{k})$. Notice now that 
$$\PP_{\mu}\big(Z_{k}\in U^{\alpha}\big)=\PP_{\mu}\big(|D_{k}-0.2|\leq \alpha\big).$$
We therefore study more precisely the law of $D_{k}$. Let $K=\inf\{k\geq0\text{ such that } G_{k}=0\}$, $K$ is the jump happening at the end of the protection against the corrosion. Eventually, denote $F(s)=s+\eta\big(e^{-\frac{s}{\eta}}-1\big)$. One has then
$$\left\{\begin{array}{ll}
D_{k}=0 & \text{ for $k\leq K$,}\\
D_{k}=D_{k-1}+R_{k}F(S_{k})& \text{ for $k>K$.}
\end{array}\right.$$
Let us now prove that for all $k$, the random variable $R_{k}F(S_{k})$ has a bounded density. Recall that $R_{k}$ has a uniform distribution on $[a_{k};b_{k}]\subset[10^{-7};10^{-5}]$ and $S_{k}$ has an exponential distribution with parameter $\lambda_{k}$. Let now $h$ be a real bounded measurable function,
\begin{align*}
\E_{\mu}[h(R_{k}F(S_{k}))]&=\int_{0}^{+\infty}\int_{a_{k}}^{b_{k}}h(\rho F(s))\frac{1}{b_{k}-a_{k}}\lambda_{k}e^{-\lambda_{k}s}d\rho ds
\end{align*}
Introduce the following transformation
$$\left\{\begin{array}{ll}
u=\rho\\
v=\rho F(s)
\end{array}\right.$$
whose Jacobian is worth $\frac{1}{u}(F^{-1})'(\frac{v}{u})$ so that
\begin{align*}
\E_{\mu}[h(R_{k}F(S_{k}))]&=\int_{0}^{+\infty}h(v)\left(\int_{a_{k}}^{b_{k}}\frac{\lambda_{k}e^{-\lambda_{k}F^{-1}(\frac{v}{u})}(F^{-1})'(\frac{v}{u})}{(b_{k}-a_{k})u}du\right) dv.
\end{align*}
Hence, we obtain the density of the random variable $R_{k}F(S_{k})$ and integration by parts yields
\begin{align*}
\int_{a_{k}}^{b_{k}}\frac{\lambda_{k}e^{-\lambda_{k}F^{-1}(\frac{v}{u})}(F^{-1})'(\frac{v}{u})}{(b_{k}-a_{k})u}du&=\frac{1}{b_{k}-a_{k}}\int_{a_{k}}^{b_{k}}u\times\frac{\lambda_{k}e^{-\lambda_{k}F^{-1}(\frac{v}{u})}(F^{-1})'(\frac{v}{u})}{u^{2}}du\\
&=\frac{1}{b_{k}-a_{k}}\left(\left[ue^{-\lambda_{k}F^{-1}(\frac{v}{u})}\right]_{a_{k}}^{b_{k}}-\int_{a_{k}}^{b_{k}}e^{-\lambda_{k}F^{-1}(\frac{v}{u})}du\right).
\end{align*}
Finally, the density of the random variable $R_{k}F(S_{k})$ is bounded by
\begin{equation*}
\left|\int_{a_{k}}^{b_{k}}\frac{\lambda_{k}e^{-\lambda_{k}F^{-1}(\frac{v}{u})}(F^{-1})'(\frac{v}{u})}{(b_{k}-a_{k})u}du\right|\leq\frac{a_{k}+b_{k}}{b_{k}-a_{k}}+1\leq \frac{2b_{k}}{b_{k}-a_{k}}\leq 2.
\end{equation*}
Let $j\in \N$, we now study the distribution of the random variables $(D_{k})_{k\in\N}$ conditionally to the event $\{K=j\}$. An induction argument provides that, conditionally to the event $\{K=j\}$, the random variable $D_{k}$ has distribution $\delta_{0}$ for $k\leq j$ and has a density $\psi_{k}$ bounded by 2 for $k>j$. Indeed, in the second case, the density of $D_{k}$ may be obtained by convolution since $D_{k-1}$ and $R_{k}F(S_{k})$ are independent random variables.\\
Therefore, for $k\leq j$, $\PP_{\mu}\big(|D_{k}-0.2|\leq\alpha\big|K=j\big)=\1_{\{\alpha\geq 0.2\}}\leq 5\alpha$ since $D_{k}=0$ for $k\leq j$ and for $k>j$, $\PP_{\mu}\big(|D_{k}-0.2|\leq\alpha\big|K=j\big)=\int_{0.2-\alpha}^{0.2+\alpha}\psi_{k}(v)dv\leq 4\alpha$ since $\psi_{k}\leq2$.
Eventually,
\begin{equation*}
\PP_{\mu}(Z_{k}\in U^{\alpha})=\PP_{\mu}(|D_{k}-0.2|\leq\alpha)
=\sum_{j\in\N}\PP_{\mu}(|D_{k}-0.2|\leq\alpha\big|K=j)\PP_{\mu}(K=j)
\leq 5 \alpha.
\end{equation*}
The results follows.\hfill$\Box$
\end{proof}

\begin{Lemma}\label{prop-u*-lip-cor} 
For all $k\in \N$, let 
$$\eta_{k}=\left|\big((T_{k}+u^{*}(Z_{k}))\wedge T_{k+1}\big)-\big((\widehat{T}_{k}+u^{*}(\widehat{Z}_{k}))\wedge \widehat{T}_{k+1}\big)\right|,$$ 
one has for all $\alpha>0$,  $$\|\eta_{k}\|_{p}\leq \|T_{k}-\widehat{T}_{k}\|_{p} + 2\|T_{k+1}-\widehat{T}_{k+1}\|_{p}+\left([u^{*}]_{\frac{\alpha}{2}}+\frac{4C_{u^{*}}}{\alpha}\right)\|Z_{k}-\widehat{Z}_{k}\|_{p}+10 C_{u^{*}}\alpha^{\frac{1}{p}}$$
where $[u^{*}]_{\alpha}=\frac{1+C_{u^{*}}+4\times 10^{5}}{10^{-7}\left(1-e^{-\frac{\alpha}{2}}\right)}$.
\end{Lemma}
\begin{proof} Let $\alpha>0$. Let $\widetilde{U}_{\alpha}=[0,0.2-\alpha]\times\{0\}\times[10^{-7};10^{-5}]\times\{0\}$. We will prove that the function $u^{*}(d,0,\rho,0)$ is Lipschitz continuous on this set. The function $u^{*}(d,0,\rho,0)$ satisfies the following equivalent equations

$$d + d_m(u^{*})=0.2 \qquad \Leftrightarrow \qquad d+\rho\left(u^{*}+\eta\left(e^{-\frac{u^{*}}{\eta}}-1\right)\right)=0.2$$

The implicit equation satisfied by $u^{*}$ yields that, on the set $\widetilde{U}_{\alpha}$, one has $u^{*}\geq \frac{\alpha}{\rho_{max}}=10^{5}\alpha$. This lower bound will be crucial to prove the Lipschitz continuity.
Let $d,d'\leq 0.2-\alpha$ and denote $u=u^{*}(d,0,\rho,0)$ and $u'=u^{*}(d',0,\rho,0)$. Notice that $d + d_m(u)=d' + d_m(u')$ because they are both worth 0.2. Consequently $\left|d_m(u)-d_m(u')\right|=\left|d'-d\right|$ and, noticing that $\eta\leq 2\times 10^{5}$ one has
\begin{align*}
|d-d'|&=\rho\left|u-u'+\eta\big(e^{-\frac{u}{\eta}}-e^{-\frac{u'}{\eta}}\big)\right|\\
&\geq \rho\left(1-e^{-\frac{u\wedge u'}{\eta}}\right)|u-u'|\\
&\geq 10^{-7}\left(1-e^{-\frac{\alpha}{2}}\right)|u-u'|
\end{align*}
that proves the Lipschitz continuity of $u^{*}$ w.r.t. $d$ on $\widetilde{U}_{\alpha}$.\\

Similarly, let $\rho,\rho'\in[10^{-7};10^{-5}]$ and denote $u=u^{*}(d,0,\rho,0)$ and $u'=u^{*}(d,0,\rho',0)$. Notice that $d+\rho\left(u+\eta\left(e^{-\frac{u}{\eta}}-1\right)\right)=d+\rho'\left(u'+\eta\left(e^{-\frac{u'}{\eta}}-1\right)\right)$ because they are both worth 0.2. Subtracting $d+\rho\left(u'+\eta\left(e^{-\frac{u'}{\eta}}-1\right)\right)$ in both terms yields $$\rho\left|u-u'+\eta\big(e^{-\frac{u}{\eta}}-e^{-\frac{u'}{\eta}}\big)\right|=|\rho-\rho'|\left|u'+\eta\left(e^{-\frac{u'}{\eta}}-1\right)\right|.$$ A lower bound for the left-hand side term has already been computed earlier while the right hand-side is easily bounded by $\big(C_{u^{*}}+4\times 10^{5}\big)|\rho-\rho'|$, since $\eta\leq 2\times 10^{5}$, so that one has
$$\big(C_{u^{*}}+4\times 10^{5}\big)|\rho-\rho'|\geq 10^{-7}\left(1-e^{-\frac{\alpha}{2}}\right)|u-u'|$$
that proves the Lipschitz continuity of $u^{*}$ w.r.t. $\rho$ on $\widetilde{U}_{\alpha}$. Eventually, for all $\alpha>0$, the function $u^{*}$ is Lipschitz continuous on $\widetilde{U}_{\alpha}$ with Lipschitz constant $[u^{*}]_{\alpha}=\frac{1+C_{u^{*}}+4\times 10^{5}}{10^{-7}\left(1-e^{-\frac{\alpha}{2}}\right)}$. \\

Let $k\in \N$, we now intend to bound $\|\eta_{k}\|_{p}$. Define, as in the proof of Lemma \ref{prop-proba-U-alpha-cor}, the random variable $K=\inf\{k\geq0\text{ such that } G_{k}=0\}$, $K$ is the jump happening at the end of the protection against the corrosion. \\

First, notice that, on the event $\{k\leq K\}$ (i.e. when the protection from corrosion is still active), one has $Z_{k}\in E_{(i,1)}$ for some $i\in\{1,2,3\}$ and since the projection defining $\widehat{Z}_{k}$ from $Z_{k}$ ensures that they are in the same mode, one has $\widehat{Z}_{k}\in E_{(i,1)}$ too. Moreover, $u^{*}(x)=+\infty$ for all $x\in E_{(i,1)}$ so that
$$\|\eta_{k}\1_{\{k\leq K\}}\|_{p}=\|\big(T_{k+1}-\widehat{T}_{k+1}\big)\1_{\{k\leq K\}}\|_{p}\leq \|T_{k+1}-\widehat{T}_{k+1}\|_{p}.$$

Furthermore, if $Z_{k}=\Delta$ where $\Delta$ denotes the cemetery state, then $\widehat{Z}_{k}=proj_{\Gamma_{k}}(Z_{k})=\Delta$ too and one has $\eta_{k}=0$ so that
\begin{align*}
\|\eta_{k}\1_{\{k> K\}}\|_{p}\leq \|\eta_{k}\1_{\{k> K\}}\1_{\{Z_{k}\not = \Delta\}}\|_{p}\leq &\|T_{k}-\widehat{T}_{k}\|_{p} + \|T_{k+1}-\widehat{T}_{k+1}\|_{p}\\
&+ \|\big(u^{*}(Z_{k})-u^{*}(\widehat{Z}_{k})\big)\1_{\{k> K\}}\1_{\{Z_{k}\not = \Delta\}}\|_{p}.
\end{align*}

Eventually, we intend to bound the last term of the previous sum and we consider therefore the event $\{k> K\}\cap\{Z_{k}\not = \Delta\}$. On the one hand, the random variables $Z_{k}$ and $\widehat{Z}_{k}$ both belong to $E_{(i,0)}$  for some $i\in\{1,2,3\}$. On the other hand, although $U_{m}=]-\infty;0.2]\times\mathbb{R}^3$ for all $m\in M$, one has actually $Z_{k}\in [0;0.2]\times\{0\}\times[10^{-7};10^{-5}]\times\R^{+}$ p.s. and, according to remark \ref{rmq-U-convex}, $\widehat{Z}_{k}\in [0;0.2]\times\{0\}\times[10^{-7};10^{-5}]\times\R^{+}$ p.s. too. Combining the two previous remark, one has $Z_{k}\in \widetilde{U}$ and $\widehat{Z}_{k}\in \widetilde{U}$ where $\widetilde{U}=[0;0.2]\times\{0\}\times[10^{-7};10^{-5}]\times\{0\}$. Finally, let $\alpha>0$ and notice that $\widetilde{U}\subset\widetilde{U}_{\alpha}\mathop{\cup}U^{\alpha}$.  One has
$$\|\big(u^{*}(Z_{k})-u^{*}(\widehat{Z}_{k})\big)\1_{\{k\geq K\}}\1_{\{Z_{k}\not = \Delta\}}\|_{p}\leq A+B$$
where 
\begin{align*}
A&=\|\big(u^{*}(Z_{k})-u^{*}(\widehat{Z}_{k})\big)\1_{\{Z_{k}\in \widetilde{U}_{\alpha}\}}\1_{\{k\geq K\}}\|_{p},\\
B&=\|\big(u^{*}(Z_{k})-u^{*}(\widehat{Z}_{k})\big)\1_{\{Z_{k}\in U^{\alpha}\}}\1_{\{k\geq K\}}\|_{p}.
\end{align*}
The term $B$ is easily bounded thanks to Lemma \ref{prop-proba-U-alpha-cor}, $B\leq 2C_{u^{*}}\PP_{\mu}(Z_{k}\in U^{\alpha})^{\frac{1}{p}}\leq 10 C_{u^{*}}\alpha^{\frac{1}{p}}$. We now turn to the term $A$ and use the Lipschitz continuity of $u^{*}$ on $\widetilde{U}_{\beta}$ for any $\beta>0$. One has
\begin{align*}
A\leq&\|\big(u^{*}(Z_{k})-u^{*}(\widehat{Z}_{k})\big)\1_{\{Z_{k}\in \widetilde{U}_{\alpha}\}}\1_{\{\widehat{Z}_{k}\in \widetilde{U}_{\frac{\alpha}{2}}\}}\1_{\{k\geq K\}}\|_{p}\\
&+\|\big(u^{*}(Z_{k})-u^{*}(\widehat{Z}_{k})\big)\1_{\{Z_{k}\in \widetilde{U}_{\alpha}\}}\1_{\{\widehat{Z}_{k}\not\in \widetilde{U}_{\frac{\alpha}{2}}\}}\1_{\{k\geq K\}}\|_{p}\\
\leq& [u^{*}]_{\frac{\alpha}{2}}\|Z_{k}-\widehat{Z}_{k}\|_{p}+2C_{u^{*}}\|\1_{\{Z_{k}\in \widetilde{U}_{\alpha}\}}\1_{\{\widehat{Z}_{k}\not\in \widetilde{U}_{\frac{\alpha}{2}}\}}\|_{p}.\\
\end{align*}
Notice now that $\1_{\{Z_{k}\in \widetilde{U}_{\alpha}\}}\1_{\{\widehat{Z}_{k}\not\in \widetilde{U}_{\frac{\alpha}{2}}\}}\leq \1_{\{|Z_{k}-\widehat{Z}_{k}|\geq \frac{\alpha}{2}\}}$ so that finally
\begin{align*}
A&\leq [u^{*}]_{\frac{\alpha}{2}}\|Z_{k}-\widehat{Z}_{k}\|_{p}+2C_{u^{*}}\left(\PP_{\mu}\big(|Z_{k}-\widehat{Z}_{k}|\geq \frac{\alpha}{2}\big)\right)^{\frac{1}{p}}\\
&\leq [u^{*}]_{\frac{\alpha}{2}}\|Z_{k}-\widehat{Z}_{k}\|_{p}+4C_{u^{*}}\frac{\|Z_{k}-\widehat{Z}_{k}\|_{p}}{\alpha}\\
\end{align*}
and the result follows. \hfill$\Box$\\
\end{proof}

\subsubsection*{The mean exit time}
Simulation results for the approximation of the mean exit time are given in Table \ref{cor moment1}. In order to have a value of reference, a Monte Carlo method ($10^{6}$ simulations) yields the value $E[\tau]_{Monte-Carlo}=526\times 10^{3}$ h.
For the first moment, the empirical convergence rate is presented on Figure \ref{rate-moment-corrosion}.
It is estimated through a regression model as $-0.38$. Remark that it is roughly the same order as the rate of convergence of the optimal quantizer (see Theorem \ref{theore}) as here the dimension is $4$.

\begin{table}[htbp]
\begin{center}
\begin{tabular}{|c||c|c|}
\hline
Points in the quantization grids &$\widehat{p}_{N,1}$ ($\times 10^{3}$ h)& relative error to $526\times 10^{3}$h\\
\hline
$20$ points&572& 8.7\% \\
\hline
$50$ points& 569 & 8.2\%\\
\hline
$100$ points& 557 & 5.9\%\\
\hline
$200$ points& 551 & 4.8\% \\
\hline
$500$ points& 539 &  2.5\%\\
\hline
\end{tabular}
\end{center}
\caption{Simulations results for the mean exit time}
\label{cor moment1}
\end{table}

\begin{figure}[hbtp]
\begin{center}
\includegraphics[scale=.4]{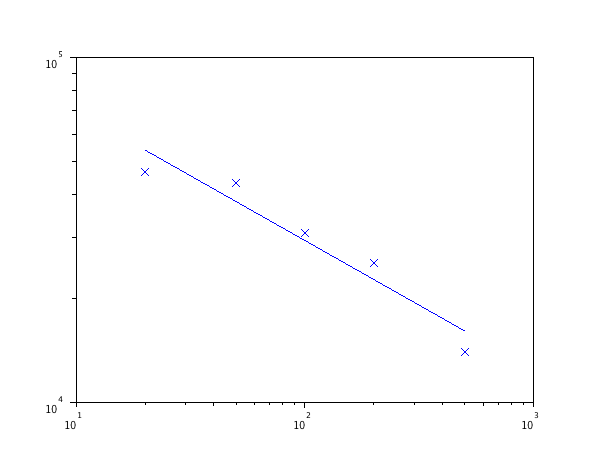}
\caption{Logarithm of the error w.r.t. the logarithm of the number of points in the quantization grids for the first moment of the corrosion process.}
\label{rate-moment-corrosion}
\end{center}
\end{figure}

\subsubsection*{The exit time distribution}

Considering the approximation scheme for the exit time distribution, one may notice that the quantized value $\widehat{p}_{N}(s)$ is not necessary smaller than 1. Therefore, it appears natural to replace $\widehat{p}_{N}(s)$ by $\widehat{p}_{N}(s)\wedge 1$. This does not change the convergence theorem and can only improve the approximation error. It is equally possible, and this is done in the results below, to replace $\widehat{p}_{N}(s)$ by $\frac{\widehat{p}_{N}(s)}{\widehat{p}_{N}(0)}$ since $\widehat{p}_{N}(0)$ goes to 1.\\

\begin{figure}[!h]
\begin{center}
\includegraphics[scale=.7]{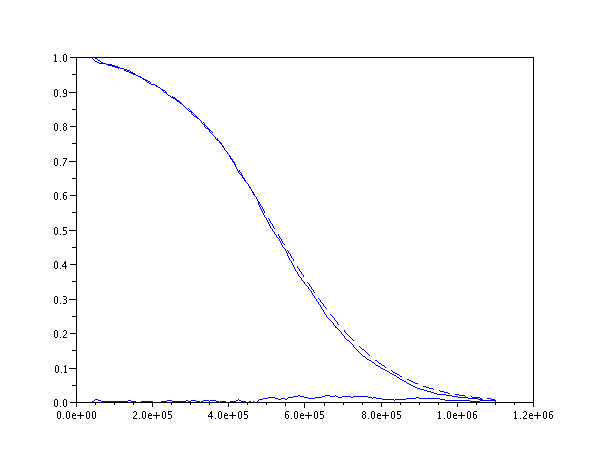}
\caption{Survival function of $\tau$ obtained through Monte Carlo simulations (dashed), quantized approximation (solid) and the error with 500 points in the quantization grids.}
\label{fig-cor}
\end{center}
\end{figure}

Figure \ref{fig-cor} presents the survival function of $\tau$ obtained through Monte Carlo simulations (the dashed line), through our approximation scheme (the solid line) and the error. Table \ref{cor distrib} contains the empirical error for different numbers of points in the quantization grids.
For the survival function, the empirical convergence rate is presented on Figure \ref{rate-survival-corrosion}.
It is estimated through a regression model as $-0.63$. Remark that it is roughly the same order as the rate of convergence of the optimal quantizer (see Theorem \ref{theore}) as here the dimension is $4$.

\begin{table}[htbp]
\begin{center}
\begin{tabular}{|c||c|}
\hline
Points in the quantization grids &$\max_{s}|p_{N}(s)-\widehat{p}_{N}(s)|$\\
\hline
$20$ points&0.145\\
\hline
$50$ points&0.119\\
\hline
$100$ points&0.040\\
\hline
$200$ points &0.039\\
\hline
$500$ points &0.020\\
\hline
\end{tabular}
\end{center}
\caption{Simulations results for the distribution}
\label{cor distrib}
\end{table}

\begin{figure}[hbtp]
\begin{center}
\includegraphics[scale=.4]{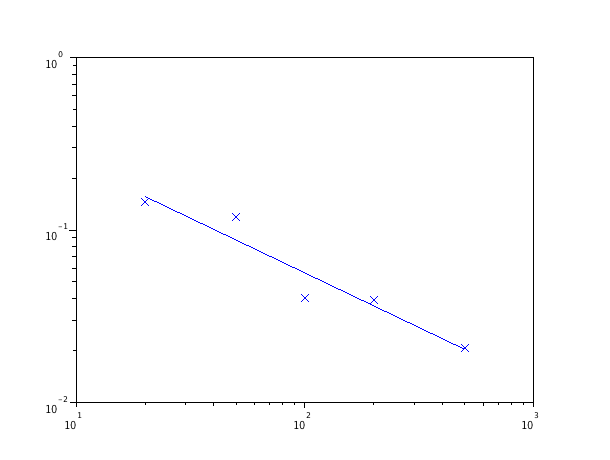}
\caption{Logarithm of the error w.r.t. the logarithm of the number of points in the quantization grids for the survival of the corrosion process.}
\label{rate-survival-corrosion}
\end{center}
\end{figure}

The convergence of the approximation scheme in the corrosion model appears to be slightly slower than in the previous example. This is due to the higher dimension of the process that has to be quantized, which is 4 in the case of the corrosion model and 1 in the case of the Poisson process.

\begin{Remark}
By using Monte Carlo simulations, we can approximate the value of~$\tilde{q}$.
One has $\tilde{q}\simeq 0.0187$ for $10^{7}$ histories.
\end{Remark}

\section{Advantages and practical interest of our approach}\label{conclusion}

Let us describe the practical interest of our approach.
\begin{itemize}
  \item {The quantizations grids only have to be computed once and for all and can be used for several purposes. Moreover, once they are obtained, the procedures leading to $\widehat{p}_{N}(s)$ and to $\widehat{p}_{N,j}$ can be achieved very simply since we only have to compute finite sums.}
  \item {Concerning the distribution, since $\widehat{p}_{N}(s)$ can be computed almost instantly for any value of $s$, the whole survival function can be obtained very quickly. Similarly, concerning the moments, $\widehat{p}_{N,j}$ can be computed very quickly for any $j$, so that any moment is almost instantly available.}
  \item {Furthermore, in both cases, one may decide to change the set $U$ and consider the exit time $\tau'$ from a new set $U'$. This will yield new sequences $(\widehat{q}_{k})_{k}$, $(\widehat{r}_{k,j})_{k}$ and $(\widehat{p}_{k,j})_{k}$ in the case of the $j$-th moment approximation or new sequences $(\widehat{q}_{k})_{k}$, $(\widehat{r}_{k}(s))_{k}$ and $(\widehat{p}_{k}(s))_{k}$ if we are interested in the distribution. These new sequences are obtained quickly and easily since the quantized process remains the same and we only have to compute finite sums. Of course, the set $U'$ must be such that Assumptions \ref{hyp_u*_lip} to \ref{hypNonRetourU} remain true and such that $\mathbf{P}_{\mu}(T_N<\tau')$ remains small without changing the computation horizon $N$. This last condition is fulfilled if, for instance, $U'\subset U$. This flexibility is an important advantage of our method over, for instance, a Monte Carlo method.}
\end{itemize}

\acks

This work was supported by ARPEGE program of the French National Agency of Research (ANR),
project ''FAUTOCOES'', number ANR-09-SEGI-004. Besides, the authors gratefully acknowledge EADS Astrium for its financial support.


\begin{thebibliography}{9}

\bibitem{quantif1}\textsc{Bally, V. and Pagès, G.} A quantization algorithm for solving multi-dimensional discrete-time optimal stopping problems. \textit{Bernoulli 9}, 6 (2003), 1003-1049.
\bibitem{quantif2}\textsc{Bally, V., Pagès, G. and Printemps, J.} A quantization tree method for pricing and hedging multidimensional American options. \textit{Math. Finance 15},1 (2005), 119-168.
\bibitem{BoutonPags}\textsc{Bouton, C. and Pagès, G.} About the multidimensional competitive learning vector quantization algorithm with constant gain. \textit{The Annals of Applied Probability}, 7 (1997), 679-710.
\bibitem{chiquet}\textsc{Chiquet, J. and Limnios, N.} A method to compute the transition function of a piecewise-deterministic Markov process with application to reliability. \textit{Statistics and Probability Letters 78}, (2008), 1397-1403.
\bibitem{davis}\textsc{Davis, M.H.A.} Markov models and optimization, vol 49 of \textit{Monograghs on Statistics and Applied Probability.} Chapman \& Hall, London, 1993.
\bibitem{arret optimal francois benoite}\textsc{de Saporta, B., Dufour, F. and Gonzalez, K.} Numerical method for optimal stopping of piecewise-deterministic Markov processes. \textit{The Annals of Applied Probability}, 20(5) (2010), 1607-1637.
\bibitem{gray}\textsc{Gray, R.M. and Neuhoff, D.L.} Quantization. \textit{IEEE Trans. Inform. Theory 44}, 6 (1998), 2325-2383. Information theory: 1948-1998.
\bibitem{}\textsc{Gugerli, U.S.} Optimal stopping of a piecewise-deterministic Markov process. \textit{Stochastics 19}, 4 (1986), 221-236.
\bibitem{LP}\textsc{Helmes, K., R\"ohl, S. and Stockbridge, R.H.} Computing moments of the exit time distribution for Markov processes by linear programming. \textit{Oper. Res}, 49 (2001), 516-530.
\bibitem{SDP}\textsc{Lasserre, J.-B. and Prieto-Rumeau, T.} SDP vs. LP relaxations for the moment approach in some performance evaluation problems. \textit{Stochastic Models}, (2004), 1-25.
\bibitem{quantif}\textsc{Pagès, G., Pham, H. and Printemps, J.} Optimal quantization methods and applications to numerical problems in finance. In \textit{Handook of computational and numerical methods in finance.} Birkhäuser Boston, Boston, MA, 2004, pp. 253-297.
\end{thebibliography}
\end{document}